\documentclass[aos,preprint,12pt]{imsart}
\usepackage[authoryear]{natbib}
\usepackage{amsmath,amssymb,fullpage,graphicx,tikz}
\usepackage{tocloft}

\let\hat\widehat
\let\tilde\widetilde

\newtheorem{theorem}{Theorem}
\newtheorem{lemma}[theorem]{Lemma}

\newtheorem{corollary}[theorem]{Corollary}

\newtheorem{remark}[theorem]{Remark}
\newenvironment{proof}{{\bf Proof.}}{$\Box$}

\newenvironment{enum}{
\begin{enumerate}
  \setlength{\itemsep}{1pt}
  \setlength{\parskip}{0pt}
  \setlength{\parsep}{0pt}
}{\end{enumerate}}

\begin{document}

\begin{center}
{\bf\sf\Large Estimating Undirected Graphs}\\ \vspace{.1cm}
{\bf\sf\Large Under Weak Assumptions}\\ \vspace{.1cm}
{\bf Larry Wasserman, Mladen Kolar and Alessandro Rinaldo}\\  \vspace{.1cm}
{\bf Carnegie Mellon University and The University of Chicago}\\  \vspace{.1cm}
{\bf September 26 2013}
\end{center}

\begin{quote}
  We consider the problem of providing nonparametric confidence
  guarantees for undirected graphs under weak assumptions.  In
  particular, we do not assume sparsity, incoherence or Normality.  We
  allow the dimension $D$ to increase with the sample size $n$.
  First, we prove lower bounds that show that if we want accurate
  inferences with low assumptions then there are limitations on the
  dimension as a function of sample size.  When the dimension
  increases slowly with sample size, we show that methods based on
  Normal approximations and on the bootstrap lead to valid inferences
  and we provide Berry-Esseen bounds on the accuracy of the Normal
  approximation.  When the dimension is large relative to sample size,
  accurate inferences for graphs under low assumptions are not
  possible.  Instead we propose to estimate something less demanding
  than the entire partial correlation graph.  In particular, we
  consider: cluster graphs, restricted partial correlation graphs and
  correlation graphs.
\end{quote}

\section{Introduction}

There are many methods for estimating undirected graphs, such as the
glasso \citep{yuan2007model, friedman2007graphical} and sparse
parallel regression \citep{meinshausen2006high}.  While these methods
are very useful, they rely on strong assumptions, such as Normality,
sparsity and incoherence, and they do not come with confidence
guarantees.  Recently, some papers --- such as \cite{WLiu} and
\cite{Ren} --- have provided confidence guarantees.  Moreover, they
have eliminated the incoherence assumption.  But they still rely on
Normality, eigenvalue conditions and sparsity.

The purpose of this paper is to construct a nonparametric estimator
$\hat G$ of an undirected graph $G$ with confidence guarantees that
does not make these assumptions.  Our approach is very traditional;
when the dimension $D_n$ is less than the sample size $n$ (but
increasing with $n$) we simply use the bootstrap or the delta method
to get confidence intervals for the partial correlations.  We put an
edge between two nodes if zero is not in the confidence interval.
When $D_n$ is larger than $n$, we avoid sparsity and eigenvalue
conditions and instead, we again rely on a more traditional method,
namely, dimension reduction.  We provide explicit Berry-Esseen style
bounds for the delta method and the bootstrap.  Indeed, while the low
dimensional case and high dimensional case have received much
attention, the moderate dimensional case --- where $D_n$ increases
with $n$ but is less than $n$ --- has not received much attention
lately.  Examples of research for increasing but moderate dimensions
are \cite{portnoy1988asymptotic} and \cite{mammen1993bootstrap}.  Our
results are very much in the spirit of those papers.  However, our
emphasis is on finite sample Berry-Esseen style bounds.

The confidence guarantee we seek is
\begin{equation}\label{eq::goal1}
  P^n(\hat G \subset G) \geq 1-\alpha - O(r_n)
\end{equation}
where $n$ is the sample size, $P^n$ denotes the distribution for $n$
observations drawn from $P$ and $r_n$ is an explicit rate.  The
notation $\hat G \subset G$ means that the edges of $\hat G$ are a
subset of the edges of $G$.  This means that, with high probability,
there are no false edges.  Of course, one could use other error
measures such as false discovery rates, but we shall use the guarantee
given by (\ref{eq::goal1}).

We focus at first on partial correlation graphs: a missing edge means
that the corresponding partial correlation is 0.  We distinguish two
cases.  In the first case, $D_n$ can increase with $n$ but is smaller
than $n$.  In that case we show that Gaussian asymptotic methods and
bootstrap methods yield accurate confidence intervals for the partial
correlations which then yield confidence guarantees for the graph.
The accuracy of the coverage is $O( \log D_n/n^{1/8})$.  We also show
that, in principle, one can construct finite sample intervals, but
these intervals turn out to be too conservative to be useful.

In the second case, $D_n$ can be large, even larger than $n$.  In this
case it is not possible to get valid inferences for the whole graph
under weak assumptions.  We investigate several ways to handle this
case including: cluster graphs, restricted partial correlation graphs
and correlation graphs.

{\em Contributions.}
We provide
graph estimation methods with these properties:
\begin{enum}
\item The methods provides confidence guarantees.
\item The methods do not depend on Normality or other parametric assumptions.
\item The methods do not require sparsity or incoherence conditions.
\item The methods have valid coverage when the dimension increases with the sample size.
\item The methods are very simple and do not require any optimization.
\item In Section \ref{section::moderate} we develop new results for
the delta method and the bootstrap with increasing dimension.
\end{enum}

{\em Related Work.}  
Our approach is similar to the method in
\cite{WLiu}, later improved by \cite{Ren}.  He uses tests on partial
correlations to estimate an undirected graph.  His approach has two
advantages over other methods: it eliminates the need to choose a
tuning parameter (as in the glasso) and it provided error control for
the estimated graph.  However, the results in that paper assume
conditions like those in most papers on the lasso, namely, sparsity.
These conditions might be reasonable in some situations, but our goal
is to estimate the graph without invoking these assumptions.  In the
special case of fixed dimension, our method is the same as that in
\cite{drton2004model}.

\cite{schafer2005shrinkage}, building on work by
\cite{ledoit2004well}, consider a shrinkage approach to estimating
graphs.  They make no sparsity or incoherence assumptions.  Their
examples suggest that their approach can work well in high dimensions.
From our point of view, their method introduces a bias-validity
tradeoff: large shrinkage biases the partial correlations but have
valid asymptotics in high dimensions.  Low shrinkage has low bias but
compromises the validity of the asymptotics in high dimensions.
Shrinkage graphs are beyond the scope of this paper, however.

{\em Outline.}  
We start with some notation in Section
\ref{section::notation}.  We discuss various assumptions in Section
\ref{section::assumptions}.  We then establish lower bounds in Section
\ref{section::lowerbounds}.  Finite sample methods are presented in
Section \ref{section::method1}.  However, these do not work well in
practice.  Asymptotic methods for the moderate dimensional case are
considered in Section \ref{section::moderate}.  Specifically, we
develop a delta method and a bootstrap method that accommodate
increasing dimension.  Recent results on high dimensional random
vectors due to \cite{Cherno,Cherno2} play an important role in our
analysis.  Methods for the high-dimensional case are considered in
Section \ref{section::methods3}.  In Section
\ref{section::experiments} we give some numerical experiments and some
examples.  Concluding remarks are in Section
\ref{section::conclusion}.

\section{Notation}
\label{section::notation}

Let $Y_1,\ldots, Y_n \in \mathbb{R}^D$ be a random sample from a
distribution $P$.  Each $Y_i = (Y_{i}(1),\ldots, Y_{i}(D))^T$ is a
vector of length $D$.  We allow $D\equiv D_n$ to increase with $n$.
We do not assume that the $Y_i$'s are Gaussian.  If $A$ is a matrix,
we will sometimes let $A_{jk}$ denote the $(j,k)$ element of that
matrix.

Let $\Sigma\equiv \Sigma(P)$ denote the $D\times D$ covariance matrix
of $Y_i$ and let $\Omega = \Sigma^{-1}$.  Let $\Theta=\{\theta\}_{jk}$
be the matrix partial correlations:
\begin{equation}
  \theta_{jk} = - \frac{\Omega_{jk}}{\sqrt{\Omega_{jj}\Omega_{kk}}}.
\end{equation}
Let 
\begin{equation}
  S_n = \frac{1}{n}\sum_{i=1}^n (Y_i - \overline{Y})(Y_i - \overline{Y})^T 
\end{equation}
be the the sample covariance matrix and let $\hat\Theta_n$ be the
matrix of sample partial correlations.  Given a matrix of partial
correlations $\Theta$ let $G\equiv G(P)$ be the undirected graph with
$D$ nodes and such that there is an edge between nodes $j$ and $k$ if
and only if $\theta_{jk}\neq 0$.  Equivalently, there is an edge if
and only if $\Omega_{jk} \neq 0$.  In Section \ref{section::methods3}
we consider other graphs.

For any matrix $A$, let ${\rm vec}(A)$ denote the vector obtained by
stacking the columns of $A$.  We define the following quantities:
\begin{align}
\mu &= \mathbb{E}(Y),\ \  & \sigma &= {\rm vec}(\Sigma), \ \  & \omega &= {\rm vec}(\Omega)\\
s   &= {\rm vec}(S_n),\ \ & \delta &= \sqrt{n}(s-\sigma),\ \  & \Delta &= \sqrt{n}(\overline{Y}-\mu).
\end{align}

If $A$ is $m\times n$ then there is a unique permutation matrix
$K_{mn}$ -- called the commutation matrix -- such that
\begin{equation}\label{eq::commutation}
K_{mn} {\rm vec}(A) = {\rm vec} (A^T).
\end{equation}
Let $J$ denote a $D\times D$ matrix of one's.
For matrices $L$ and $U$ with the same dimensions,
we write
$L \leq U$ to mean that
$L_{jk} \leq U_{jk}$ for all $j,k$.
If $A$ is $m\times n$ and $B$ is $p\times q$ then
the Kronecker product $A\otimes B$ is the $mp\times nq$ matrix
\begin{equation}
\left[
\begin{array}{ccc}
A_{11} B & \cdots & A_{1n}B\\
\vdots   &        & \vdots\\
A_{m1}B  & \cdots & A_{mn}B
\end{array}
\right].
\end{equation}
The Frobenius norm of $A$ is
$||A||_F = \sqrt{\sum_{j,k} A_{jk}^2}$,
the operator norm is
$||A|| = \sup_{||x||=1} ||Ax||$
and the max norm is
$||A||_{\rm max} = \max_{j,k} |A_{jk}|$.
Let
$||A||_1 = \max_j \sum_{i=1}^D |A_{ij}|$
and
\begin{equation}
|\!|\!|A|\!|\!| = \sum_{jk} |A_{jk}|.
\end{equation}
We let $\Phi$ denote the cdf of a standard Normal random variable.
Recall that a random vector $X\in\mathbb{R}^k$ is {\em sub-Gaussian}
if there exists $\zeta>0$ such that, for all $t\in\mathbb{R}^k$,
\begin{equation}
  \mathbb{E} e^{t^T(X-\mu)} \leq e^{||t||^2\,\zeta^2/2}
\end{equation}
where $\mu=\mathbb{E}(X)$.  The smallest and largest eigenvalues of a
matrix $A$ are denoted by $\lambda_{\rm min}(A)$ and $\lambda_{\rm
  max}(A)$.  We write $a_n \preceq b_n$ to mean that there is some
$c>0$ such that $a_n \leq c b_n$ for all large $n$.  We often use $C$
to denote a generic positive constant.

\section{Assumptions}
\label{section::assumptions}

In this section we discuss the assumptions we make and we also discuss
some of the commonly used assumptions that we will not use.

\vspace{1cm}

{\bf The Assumptions.}
In the case where $D_n < n$ we make the following assumptions:

(A1) $Y$ and ${\rm vec}(Y Y^T)$ are sub-Gaussian.

(A2) $0 < a \leq \lambda_{\rm min}(\Sigma) \leq \lambda_{\rm max}(\Sigma) \leq A < \infty$.

(A3) $\lambda_{\rm min}(T) \geq c_0 > 0$ where $T$
is the asymptotic covariance of $\sqrt{n}(s-\sigma)$
and is given in Equation (\ref{eq::T}).
Also assume that
$\min_j \gamma_{jj}>0$
where $\gamma$,
the asymptotic variances of the sample partial correlations,
is given in (\ref{eq::Gamma}).

(A4) $\max_j \mathbb{E} |V_i(j)|^3 \leq C$
where
$V_i = {\rm vec}[(Y_i-\mu)(Y_i-\mu)^T] - \sigma$.

In the case where $D_n > n$ we do not make these assumptions.
Indeed, (A3) requires that $D_n < n$.
Instead, when $D_n > n$, we first perform a dimension reduction and then we assume
(A1)-(A4) on the reduced problem.  We remark that the sub-Gaussian
assumption is stronger than needed and is made for simplicity.

{\bf The Non-Assumptions.}
Now we discuss the assumptions that are commonly made
for this problem, but that we will not use.

(B1) {\em Normality.} $Y\sim N(\mu,\Sigma)$.

(B2) {\em Incoherence.} The incoherence condition is
\begin{equation}
||\Gamma_{S^c S} (\Gamma_{S S})^{-1}||_\infty < 1
\end{equation}
where
$\Gamma = \Sigma \otimes \Sigma$,
$S$ is the set of pairs with edges between them
and $||\cdot ||_\infty$ is the maximum absolute column sum.

(B3) {\em Sparsity.}  
The typical sparsity assumption is that the
maximum degree $d$ of the graph is $o(\sqrt{n})$.

(B4) {\em Eigenvalues.} 
$0 < a \leq \lambda_{\rm min}(\Sigma) \leq \lambda_{\rm max}(\Sigma) \leq A < \infty$.

(B5) {\em Donut.}
It is assumed that each partial correlation is either 0 or is strictly
larger than $\sqrt{\log D/n}$, 
thus forbidding a donut around the origin.

{\bf Discussion.}  
The above assumptions may be reasonable in certain
specialized cases.  However, for routine data-analysis, we regard
these assumptions with some skepticism when $D_n > n$.  They serve to
guarantee that many high-dimensional methods will work, but seem
unrealistic in practice.  Moreover, the assumptions are very fragile.
The incoherence assumption is especially troubling although \cite{Ren}
have been able to eliminate it.  The donut assumption ensures that
non-zero partial correlations will be detected with high probability.
The eigenvalue assumption (B4) is quite reasonable when $D_n < n$.
But when $D_n$ is much larger than $n$, (B4) together with (B3) are
very strong and may rule out many situations that occur in real data
analysis practice.  To the best of our knowledge, (B3) and (B4) are
not testable when $D_n > n$.  Our goal is to develop methods that
avoid these assumptions.  Of course, our results will also be weaker
which is the price we pay for giving up strong assumptions.  They are
weaker because we only are able to estimate the graph of a
dimension-reduced version of the original problem.

\section{Lower Bounds}
\label{section::lowerbounds}

Constructing a graph estimator for which (\ref{eq::goal1}) holds is
easy: simply set $\hat G$ to be identically equal to the empty graph.
Then $\hat G$ will never contain false edges.  But to have a useful
estimator we also want to have non-trivial power to detect edges;
equivalently, we want confidence intervals for the partial
correlations to have width that shrinks with increasing sample size.
In this section we find lower bounds on the width of any confidence
interval for partial correlations.  This reveals constraints on the
dimensions $D$ as a function of the sample size $n$.  Specifically, we
show (without sparsity) that one must have $D_n < n$ to get consistent
confidence intervals.  This is not surprising, but we could not find
explicit minimax lower bounds for estimating partial correlations so
we provide them here.

The problem of estimating a partial correlation is intimately related
to the problem of estimating regression coefficients.  Consider the
usual regression model
\begin{equation}\label{eq::regression}
Y = \beta_1 X_1 + \ldots + \beta_D X_D + \epsilon
\end{equation}
where $\epsilon \sim N(0,\sigma^2)$ and
where we take the intercept to be 0 for simplicity.
(Normality is assumed only in this section.)
Suppose we want a confidence interval for $\beta_1$.

We will need assumptions on the covariance matrix $\Sigma$ for
$X=(X_1,\ldots, X_d)$.  Again, since we are interested in the low
assumption case, we do not want to impose strong assumptions on
$\Sigma$.  In particular, we do not want to rule out the case where
the covariates are highly correlated.  We do, however, want $\Sigma$
to be invertible.  Let ${\cal S}$ denote all symmetric matrices and
let
\begin{equation}
{\cal S}(a,A) = \Bigl\{ \Sigma \in {\cal S}:\ 
a \leq \lambda_{\rm min}(\Sigma) \leq \lambda_{\rm max}(\Sigma) \leq A\Bigr\}
\end{equation}
where $0 < a \leq A < \infty$.  To summarize: $Y = \beta^T X +
\epsilon$ where $\epsilon \sim N(0,\sigma^2)$, and $\Sigma = {\rm
  Cov}(X)\in {\cal S}(a,A)$.  Let ${\cal P}$ be all such
distributions.

A set-valued function $C_n$
is a $1-\alpha$ confidence interval for $\beta_1$ if
\begin{equation}
P^n(\beta_1 \in C_n) \geq 1-\alpha
\end{equation}
for all
$P\in {\cal P}$.
Let ${\cal C}_n$ denote all $1-\alpha$ confidence intervals.
Let
\begin{equation}
W_n = \sup \{x:\ x\in C_n\} - \inf \{x:\ x\in C_n\}
\end{equation}
be the width of $C_n$.

\begin{theorem}\label{theorem::regression}
Assume that $D_n < n-D-1$ and that
$\alpha < 1/3$.
Then
\begin{equation}
\inf_{C_n \in {\cal C}_n}\sup_{P\in {\cal P}}
\mathbb{E}(W_n^2) \geq \frac{C}{n-D+1}
\end{equation}
for some $C>0$.
\end{theorem}

\begin{proof}
Let us write the model in vectorized form:
\begin{equation}\label{eq::matrix}
Y = X \beta + \epsilon
\end{equation}
where
$Y=(Y_1,\ldots,Y_n)^T$,
$X$ is $n\times D$,
$\beta = (\beta_1,\ldots,\beta_D)^T$ and
$\epsilon = (\epsilon_1,\ldots,\epsilon_n)^T$.

Let $M = N(0,\Sigma)$
with $\lambda_{\rm min}(\Sigma) \geq a > 0$.
Let
$p_0(x,y) = p_0(y|x) m(x)$ and
$p_1(x,y) = p_1(y|x) m(x)$ 
where $p_0(y|x)$ and $p_1(y|x)$
will be specified later.
Now
\begin{align*}
\inf_{C_n\in {\cal C}_n}\sup_{P\in {\cal P}}\mathbb{E}(W_n) & \geq
\inf_{C_n\in {\cal C}_n}\max_{P\in {P_0,P_1}}\mathbb{E}(W_n)\\
&= \inf_{C_n\in {\cal C}_n}\max_{P\in {P_0,P_1}}\int \mathbb{E}(W_n|X=x)dM(x)\\
&= \inf_{C_n\in {\cal C}_n} \max_{j=0,1}\int R_j(x) dM(x)
\end{align*}
where
$R_j(x) = \mathbb{E}_j(W_n|X=x)$.
Let
$$
A = \Bigl\{x :\ R_0(x) > R_1(x)\Bigr\}.
$$
For any two real numbers $r_0,r_1$,
we have that
$\max\{r_0,r_1\} \geq (r_0+r_1)/2$.
Hence,
\begin{align*}
\int R_0(x)dM(x) &\vee \int R_1(x)dM(x) \geq
\int_{A} R_0(x)dM(x)   \vee \int_{A^c} R_1(x) dM(x) \\
&= \int_{A} [R_0(x)\vee R_1(x)] dM(x) \vee 
\int_{A^c} [R_0(x)\vee R_1(x)] dM(x)\\
& \geq
\frac{1}{2}
\left(\int_{A} [R_0(x)\vee R_1(x)] dM(x) +
\int_{A^c} [R_0(x)\vee R_1(x)] dM(x)\right)\\
&= \frac{1}{2}\int [R_0(x)\vee R_1(x)] dM(x).
\end{align*}
Hence,
\begin{align*}
\inf_{C_n\in {\cal C}_n}\sup_{P\in {\cal P}}\mathbb{E}(W_n) & \geq
\inf_{C_n}\frac{1}{2}\int [\mathbb{E}_0(W_n|X=x)\vee \mathbb{E}_1(W_n|X=x)] dM(x)\\
& \geq
\frac{1}{2}\int \inf_{C_n} \max_{P_0,P_1}\mathbb{E}_P(W_n|X=x) dM(x).
\end{align*}

Now we fix $X=x\in\mathbb{R}^{n\times D}$ and lower bound
$\inf_{C_n} \max_{P_0,P_1}\mathbb{E}_P(W_n|X=x)$.
Assume that $x^T x$ is invertible.
Consider Equation (\ref{eq::matrix})
where the matrix $X$ is taken as fixed.
Multiplying each term in the equation by
$(x^T x)^{-1} x^T$
we can rewrite the equation as
$$
Z = \beta + \xi
$$
where, given $X=x$, $\xi \sim N(0,(x^T x)^{-1})$.

Let $S=x^T x$,
$b>0$,
$\delta^2 = 4\alpha^2 S_{11}^{-1}$,
$\beta_0 = (0,b,\ldots, b)$ and
$\beta_1 = (\delta,b,\ldots, b)$
which now defines $P_0$ and $P_1$.
The (conditional) Kullback-Leibler
distance between $p_0(y|x)$ and $p_1(y|x)$ is
$$
\frac{1}{2}(\beta_1-\beta_0)^T (x^T x) (\beta_1-\beta_0) = 2\alpha^2 S_{11}^{-1} S_{11}.
$$
Note that, since $D < n-1$, $x^T x$ is invertible with probability one.
The conditional total variation distance is thus bounded above by
${\rm TV}(x) \equiv \alpha \sqrt{S_{11}^{-1} S_{11}}$.
Let
$A_0 = \{ 0\in C_n\}$ and
$A_1 = \{ \delta\in C_n\}$.
Note that
$A_0 \cap A_1$ implies that
$W_n^2 \geq \delta^2$.
So, given $X=x$,
\begin{align*}
P_0(W_n^2 \geq \delta^2|X=x) &\geq
P_0(A_0\cap A_1|X=x)\\
& =
P_0(A_0|X=x)+ P_0(A_1|X=x) - P_0(A_0\cup A_1|X=x)\\
& \geq
P_0(A_0|X=x)+ P_0(A_1|X=x) - 1\\
& \geq
P_0(A_0|X=x)+ P_1(A_1|X=x) - 1 - {\rm TV}(x).
\end{align*}
Note that
$\int {\rm TV}(x) dM(x) \leq \alpha \int \sqrt{S_{11}^{-1} S_{11}} dM(X)$.
Now 
$\int\sqrt{S_{11}^{-1} S_{11}} dM(x) \to 1$ as $n\to \infty$.
Thus, for large enough $n$,
$\int {\rm TV}(x) dM(x) \leq 2\alpha$.
Integrating over $dM(x)$ we have
\begin{align*}
P_0(W_n^2 \geq \delta^2) &\geq
P_0(A_0)+ P_1(A_1) - 1 - \int {\rm TV}(x) dM(x)\\
& \geq
[1-\alpha] + [1-\alpha] - 1 - 2\alpha = 1 -  4\alpha.
\end{align*}
Let
$E=\left\{ S^{-1}_{11} \geq \frac{C}{n-D+1}\right\}$
where $C$ is a small positive constant.
Then,
\begin{align*}
P_0(W_n^2 \geq \delta^2) &=
P_0\left(W_n^2 \geq 4\alpha^2 S^{-1}_{11}\right)\\
&= P_0\left(W_n^2 \geq 4\alpha^2 S^{-1}_{11},E\right) + P_0\left(W_n^2 \geq 4\alpha^2 S^{-1}_{11},E^c\right)\\
&\leq  P_0\left(W_n^2 \geq \frac{4C\alpha^2}{n-D+1}\right) + P_0\left(E^c\right).
\end{align*}
Recalling that $C$ is a small positive constant,
\begin{align*}
P_0(E^c) &= P_0\left( S^{-1}_{11} < \frac{C}{n-D+1}\right) =
P_0\left( \frac{1}{S^{-1}_{11}} > \frac{n-D+1}{C}\right)\\
&= P_0\left( \chi^2_{n-D+1} > \frac{n-D+1}{C}\right) < \frac{1}{n}.
\end{align*}
So
$$
P_0\left(W_n^2 \geq \frac{4C\alpha^2}{n-D+1}\right) \geq
P_0(W_n^2 \geq \delta^2)  - \frac{1}{n} \geq 1-4\alpha - \frac{1}{n}.
$$
By Markov's inequality,
$$
E_0(W_n^2)  \geq \left(1-4\alpha-\frac{1}{n}\right) \frac{4C \alpha^2}{n-D+1} \succeq \frac{1}{n-D+1}.
$$
\end{proof}

Now we establish the analogous upper bound.

\begin{theorem}\label{theorem::upper}
Assume that $D_n < n-D+1$ and that
$\alpha < 1/3$.
Then
\begin{equation}
\inf_{C_n \in {\cal C}_n}\sup_{P\in {\cal P}}
\mathbb{E}(W_n^2) \preceq \frac{C}{n-D+1}.
\end{equation}
\end{theorem}

\begin{proof}
We derive a sharp $\ell_\infty$ bound on $\hat\beta - \beta$.
Consider the following model
\[
Y = X \beta + \epsilon
\]
where $Y \in \mathbb{R}^n$, $X \in \mathbb{R}^{n \times D}$ are jointly Gaussian. In particular, 
$x_i \sim N(0, \Sigma)$ and $\epsilon_i \sim N(0, \sigma_{\epsilon}^2)$.
The OLS estimator is
\begin{align*}
\hat \beta 
& = \beta + (X^T X)^{-1}X\epsilon  = \beta + Z.
\end{align*}
Since $Z|X \sim N(0, \sigma_{\epsilon}^2(X'X)^{-1})$, we have that 
\[
|Z_j| \leq \sqrt{\sigma_{\epsilon}^{-1} (X^T X)^{-1}_{jj} \log(2\alpha^{-1})}
\]
with probability $1-\alpha/2$, conditional on $X$. 
We have that $X^T X \sim W_D(\Sigma, n)$ and
\[
\frac{\Sigma_{jj}^{-1}}{(X'X)_{jj}^{-1}} \sim \chi^2_{n-D+1}.
\]
For $T \sim \chi_D^2$, we have
\[
P(|D^{-1}T - 1|\geq x )
 \leq \exp^{-\frac{3}{16}Dx^2}.
\]
Therefore, setting $x = \sqrt{\frac{16}{3}\frac{\log(2\alpha^{-1})}{n-D+1}}$,
\[
(X'X)_{jj}^{-1} \leq \frac{\Sigma_{jj}^{-1}}{(1-x)(n-D+1)}
\]
with probability $1-\alpha/2$.
Combining the results, we have that for $j \in [p]$, 
\begin{align*}
|Z_j| 
&\leq 
\sqrt{\frac{\sigma_{\epsilon}^{2} (1-x)^{-1}\Sigma_{jj}^{-1} \log(2\alpha^{-1})}{n-D+1}}
\end{align*}
with probability $1-\alpha/2$. The second inequality hold under the
assumption that $D = o(n)$.  Using the lower quantile we obtain a
$\alpha/2$ level lower bound.  This yields a confidence interval with
squared length of order $O(1/(n-D+1))$.
\end{proof}

Now consider estimating a partial correlation corresponding to a
covariance matrix $\Sigma$.

\begin{theorem}
Let $W\in \mathbb{R}^D$ where $W\sim N(0,\Sigma)$ with $\Sigma\in
{\cal S}_a$.  Let $\theta$ be the partial correlation between two
components of $W$, say, $W_{D}$ and $W_{D-1}$.  Let ${\cal C}_n$ be the
set of $1-\alpha$ confidence intervals for $\theta$.  Assume that
$D_n\leq n$ and that $\alpha < 1/4$.  Then
\begin{equation}
\inf_{C_n \in {\cal C}_n}\sup_{P\in {\cal P}}
\mathbb{E}(W_n^2) \geq \frac{C}{n-D+1}.
\end{equation}
\end{theorem}

\begin{proof}
Let $b>0$ be a small positive constant.
Let $W=(W_1,\ldots, W_D)$ where
\begin{align*}
W_1 &= \epsilon_1\\
W_2 &= bW_1 + \epsilon_2\\
W_3 &= bW_2 + bW_1 + \epsilon_3\\
\vdots &= \vdots\\
W_D &= q W_{D-1} + b W_{D-2} + \cdots + b W_1 + \epsilon_D,
\end{align*}
$\epsilon_1,\ldots, \epsilon_D \sim N(0,1)$.
For $P_0$ take $q=0$
and for $P_1$ take $q = \delta$.
So, $P_0 = N(0,\Sigma_0)$ and
$P_1 = N(0,\Sigma_1)$, say.
Then $\Omega_1=\Sigma^{-1}_1$
corresponds to a complete graph while
$\Omega_0=\Sigma_0^{-1}$
has a missing edge.
See Figure \ref{fig::two-graphs}.
Let us write
$W=(Y,X)$ where
$Y=W_1$ and
$X = (W_2,\ldots,W_D)$.
We note that
the marginal distribution of
$X$ is the same under $P_0$ and $P_1$.
The conditional distribution of $Y$ given $X$ under $P_j$ can be written
$$
Y = \beta^T_j X + \epsilon
$$
where
$\beta_0 = (0,b,\ldots, b)$ and
$\beta_1 = (\delta,b,\ldots, b)$.
The rest of the proof follows
the proof of Theorem \ref{theorem::regression}.
\end{proof}

We conclude that without further assumptions
(namely sparsity plus incoherence)
we cannot make reliable inferences unless
$D < n$.

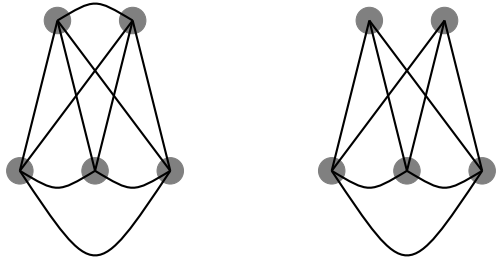
\begin{figure}

\begin{tabular}{ccc}
\begin{tikzpicture}
\filldraw [gray] (0,0) circle (5pt)
 (1,0) circle (5pt)
 (2,0) circle (5pt)
 (.5,2) circle (5pt)
 (1.5,2) circle (5pt);
\draw[thick] (.5,2) -- (0,0);
\draw[thick] (.5,2) -- (1,0);
\draw[thick] (.5,2) -- (2,0);

\draw[thick] (1.5,2) -- (0,0);
\draw[thick] (1.5,2) -- (1,0);
\draw[thick] (1.5,2) -- (2,0);

\draw[thick] (0,0) .. controls (1,-1.5) .. (2,0);
\draw[thick] (0,0) .. controls (.5,-.3) .. (1,0);
\draw[thick] (1,0) .. controls (1.5,-.3) .. (2,0);
\draw[thick] (.5,2) .. controls (1,2.3) .. (1.5,2);
\end{tikzpicture} & \ \ \ \ \ \ \ \ \ &
\begin{tikzpicture}
\filldraw [gray] (0,0) circle (5pt)
 (1,0) circle (5pt)
 (2,0) circle (5pt)
 (.5,2) circle (5pt)
 (1.5,2) circle (5pt);
\draw[thick] (.5,2) -- (0,0);
\draw[thick] (.5,2) -- (1,0);
\draw[thick] (.5,2) -- (2,0);

\draw[thick] (1.5,2) -- (0,0);
\draw[thick] (1.5,2) -- (1,0);
\draw[thick] (1.5,2) -- (2,0);

\draw[thick] (0,0) .. controls (1,-1.5) .. (2,0);
\draw[thick] (0,0) .. controls (.5,-.3) .. (1,0);
\draw[thick] (1,0) .. controls (1.5,-.3) .. (2,0);
\end{tikzpicture}
\end{tabular}
\caption{
The two graphs in the proof. Left: $\Omega_1$
corresponds to a dense graph. 
Right: $\Omega_0$ is the same 
as $\Omega_1$ except that 
an edge has been dropped.}
\label{fig::two-graphs}
\end{figure}

\begin{remark}
These lower bounds were computed under the assumption of Normality.
This is good enough to show the dependence on dimension.
However, this makes the minimax lower bound optimistic.
When we develop the methods,
we shall not assume Normality.
\end{remark}

\section{A Finite Sample Method}
\label{section::method1}

For completeness, we give here a finite sample confidence interval
that has length $O(\sqrt{D/n})$.  However, the intervals do not work
well in practice and we explore asymptotic methods in the following
section.  In this section we suppose that $|Y_{ij}|\leq B$ for some
finite constant $B$.  First we recall the following result from
\cite{vershynin2010introduction}.

\begin{theorem}[Vershynin 2010]
There exists
$c_\alpha$, depending only on $B$, such that
$$
P^n\left( ||S-\Sigma|| > c_\alpha \sqrt{\frac{D}{n}}\right) \leq \alpha.
$$
\end{theorem}

\begin{theorem}
Let
\begin{equation}
\epsilon_n = \frac{c_\alpha}{\hat\lambda^2} \sqrt{\frac{D}{n}} 
\left(1 - \frac{c_\alpha}{\hat\lambda} \sqrt{\frac{D}{n}}\right)^{-1}
\end{equation}
where $\hat\lambda$ is the smallest eigenvalue of $S_n$.  Let
$\Delta_n = 2\epsilon_n/(1-\epsilon_n)$.  Then
\begin{equation}
\inf_{P\in {\cal P}}P^n( \underline{\Theta} \leq \Theta \leq \overline{\Theta}) \geq 1-\alpha
\end{equation}
where $\overline{\Theta} = \hat\Theta + \Delta_n J$ and
$\underline{\Theta} = \hat\Theta - \Delta_n J$ where we recall that
$J$ is a $D\times D$ matrix of one's.
\end{theorem}

\begin{proof}
  By the previous result, $ ||S-\Sigma|| \leq c_\alpha
  \sqrt{\frac{D}{n}}$ with probability at least $1-\alpha$.  From
  \cite{horn1990matrix} page 381,
$$
||S^{-1}-\Sigma^{-1}||_{\rm max}  \leq
\frac{ ||S^{-1}||\ ||S^{-1} (\Sigma-S)||}{ 1-  ||S^{-1} (\Sigma-S)||}.
$$
Note that, with probability at least $1-\alpha$,
\begin{equation}\label{eq::abound}
||S^{-1} (\Sigma-S)|| \leq
||S^{-1}||\, ||\Sigma-S|| =
\frac{||\Sigma-S||}{\hat\lambda} \leq
\frac{c_\alpha}{\hat\lambda} \sqrt{\frac{D}{n}}.
\end{equation}
Also note that
$||S^{-1}|| \leq 1/\hat\lambda$.
We conclude that
$$
||S^{-1}-\Sigma^{-1}||_{\rm max}  \leq \epsilon_n.
$$
From Lemma 3 of
\cite{harris2012pc},
$||\hat\Theta - \Theta||_{\rm max} \leq \frac{2\delta}{1-\delta}$
where
$\delta = ||S^{-1}-\Sigma^{-1}||_{\rm max}.$
The result follows.
\end{proof}

Despite the apparent optimal rate, in practice the confidence
intervals are gigantic.  Instead, we turn to asymptotic methods.

\section{Increasing Dimension}
\label{section::moderate}

We call the case where $D_n$ is increasing with $n$ but smaller than
$n$, the moderate dimensional case.  Here we derive confidence sets
for the partial correlations in this case.  We deal with the
high-dimensional case $D_n > n$ in the next section.

Our goal is to show the accuracy of the delta method and the
bootstrap.  In particular, we develop new results on the delta method
for multiple non-linear statistics with increasing dimension.  The
state-of-the-art for delta method results are the papers by
\cite{Pinelis, chen2007normal} where, in particular, the former
applies to the multivariate case.  Rather than adapt those results, we
instead develop a slightly different approach that leverages recent
developments in high dimensional statistics.  This allows us to
develop a simultaneous delta method and bootstrap for multiple inference with
increasing dimension.  Throughout this section, we assume that $D_n <
n$.

\subsection{Preliminary Definitions and Results}

Recall that
$s = {\rm vec}(S)$, $\sigma = {\rm vec}(\sigma)$,
$\omega = {\rm vec}(\Omega)$,
$\theta = {\rm vec}(\Theta)$
and
$\delta = \sqrt{n}(s - \sigma)$.
Define the map $g_j$ by
$\theta_j = g_j(\sigma)$.
We can write
$\theta = G(\sigma)$
where
$G(\sigma) = (g_1(\sigma),\ldots, g_{D^2}(\sigma))^T$.
Note that
$G: \mathbb{R}^{D^2} \to \mathbb{R}^{D^2}$.

If $D$ is fixed, the
central limit theorem implies that
\begin{equation}
\sqrt{n}(s - \sigma) \rightsquigarrow  N(0,T)
\end{equation}
where
\begin{equation}\label{eq::T}
T  \equiv T(\sigma) = \mathbb{E}( \epsilon \epsilon^T \otimes \epsilon \epsilon^T )  - \sigma \sigma^T 
\end{equation}
and $\epsilon \sim N(0,\Sigma)$.
The finite sample variance matrix of
$\delta$ is given by
(\cite{boik2006second}),
\begin{equation}
T_n(\sigma) =
\frac{c_1}{n-1} (\mathbb{E}( \epsilon \epsilon^T \otimes \epsilon \epsilon^T )  - \sigma \sigma^T)+
\left(1 - \frac{D\left(1 - \frac{1}{n}\right)}{n-1}\right) (I_{D^2} - K_{(D,D)}) (\Sigma\otimes\Sigma)
\end{equation}
where $K_{(D,D)}$
is the commutation matrix defined in
(\ref{eq::commutation}) and
$c_1 = D\left(1 - \frac{1}{n}\right)$.

Let
$\tilde{S} = n^{-1}\sum_{i=1}^n (Y_i-\mu)(Y_i-\mu)^T$,
$\tilde{s} = {\rm vec}(\tilde S)$,
$Q = (\overline{Y}-\mu)(\overline{Y}-\mu)^T$ and
$q = {\rm vec}(Q)$.
Note that
\begin{equation}\label{eq::qsigma}
s-\sigma = \tilde{s} - \sigma - q = \overline{V}-q
\end{equation}
where
$\overline{V} = n^{-1} \sum_i V_i$
and $V_i = {\rm vec}( (Y_i-\mu)(Y_i-\mu)^T)-\sigma$.

\begin{lemma}\label{lemma::ballbound}
For all $\epsilon>0$ we have
the following inequalities:
\begin{align*}
P(||s-\sigma||_\infty > \epsilon) &\leq 2 D^2 e^{-n \zeta^2 \epsilon^2/2}\\
P(||s-\sigma|| > \epsilon) &\leq  2D^2 e^{-n \zeta^2 \epsilon^2/(2D^2)}\\
\mathbb{E} ||\delta||_\infty &\leq \zeta \sqrt{2 \log(2 D^2)}\\
P(||q||_\infty > \epsilon) & \leq   4 D^2 e^{-n \epsilon \zeta^2/2}.
\end{align*}
\end{lemma}

\begin{proof}
Using the sub-Gaussian property, we have
\begin{align*}
P( ||s-\sigma||_\infty > \epsilon) &= P( ||\overline{V}||_\infty > \epsilon) \leq 
\sum_j P( |\overline{V}_j| > \epsilon)
 \leq 2\sum_j e^{-n \zeta^2 \epsilon^2/2} = 2D^2 e^{-n \zeta^2 \epsilon^2/2}.
\end{align*}
The second result follows from the first since $||s-\sigma|| \leq D
||s-\sigma||_\infty$.  The third inequality follows from a standard
inequality; see Lemma 2.2 of Devroye and Lugosi (2001) for example.
For the fourth inequality, note that the absolute value $|q_j|$ of
each element of $q$ has the form $|\overline{Y}(s) - \mu(s)|\
|\overline{Y}(t) - \mu(t)|$.  So $P(||q||_\infty > \epsilon) \leq
\sum_j P( |q_j| > \epsilon) \leq 4 D^2 e^{-n \epsilon \zeta^2/2}.$
\end{proof}

\begin{lemma}\label{lemma::simple}
Let $Z\sim N(0,1)$.
Then, for every $\epsilon>0$,
$$
\sup_z | P(A_n + B_n < z) - \Phi(z)| \leq
\sup_z | P(A_n < z) - \Phi(z)| + \epsilon + P(|B_n| > \epsilon).
$$
\end{lemma}

\begin{proof}
Let
$E = \Bigl\{ |B_n| < \epsilon\Bigr\}$.
Then
\begin{align*}
P(A_n + B_n < z) - \Phi(z) &=
P(A_n + B_n < z,E) + P(A_n + B_n < z,E^c) - \Phi(z)\\
& \leq P(A_n  < z+\epsilon) + P(E^c) - \Phi(z)\\
& \leq P(A_n  < z+\epsilon) - \Phi(z+\epsilon) - \Phi(z) + \Phi(z+\epsilon) + P(|B_n| > \epsilon)\\
& \leq P(A_n  < z+\epsilon) - \Phi(z+\epsilon) + \epsilon + P(|B_n| > \epsilon).
\end{align*}
Hence,
\begin{align*}
\sup_z[P(A_n + B_n < z) - \Phi(z)] &\leq
\sup_z[P(A_n  < z+\epsilon) - \Phi(z+\epsilon)] + \epsilon + P(|B_n| > \epsilon)\\
&=
\sup_z[P(A_n  < z) - \Phi(z)] + \epsilon + P(|B_n| > \epsilon).
\end{align*}
By a similar argument,
$$
\sup_z[P(A_n + B_n < z) - \Phi(z)] \geq
\sup_z[P(A_n  < z) - \Phi(z)] -  \epsilon - P(|B_n| > \epsilon).
$$
\end{proof}

We need the following recent results on high-dimensional
random vectors.

\begin{theorem}[High-Dimensional CLT; Chernozhukov, Chetverikov and Kato 2012]
\label{theorem::clt}
Let $Y_1,\ldots, Y_n\in \mathbb{R}^k$
be random vectors with mean $\mu$ and 
covariance $\Sigma$.
Let
$$
T = \max_j \Biggl|\frac{1}{\sqrt{n}}\sum_{i=1}^n (Y_i(j) - \mu(j)) \Biggr|.
$$
Let $Z\in\mathbb{R}^D$ be Gaussian with mean 0 and
covariance $\Sigma$.
Then
\begin{equation}
\sup_z \Biggl| P( T\leq z) - P(\max_j |Z_j| \leq z)\Biggr| \preceq M \frac{ (\log D)^{7/8}}{n^{1/8}}
\end{equation}
where
$M = \left(E \max_j [|Y(j)| + |Z(j)|]^3\right)^{1/4}$.
Under the sub-Gaussian assumption, $M \preceq (\log D)^{1/8}$.
Hence the upper bound is
$\log D/n^{1/8}$.
\end{theorem}

\begin{theorem}[Gaussian Anti-Concentration; Chernozhukov, Chetverikov and Kato 2013]
Let $Z_1,\ldots, Z_k$
be centered, not necessarily independent,
Gaussian random variables.
Then
\begin{equation}
\sup_z P\Biggl( |\max_j Z_j -z| \leq \epsilon\Biggr) \leq C \epsilon \sqrt{\log (k/\epsilon)}
\end{equation}
where $C$ depends only on
$\max_j {\rm Var}(Z_j)$ and
$\min_j {\rm Var}(Z_j)$.
\end{theorem}

An immediate corollary of this result is the following.

\begin{lemma}\label{lemma::anticoncentration}
Let $Z\sim N(0,\Sigma)$.
There exists $c>0$ depending only on 
$\max_j \Sigma_{jj}$ 
and
$\min_j \Sigma_{jj}$ 
but not on $k$ such that,
for every $\epsilon>0$,
$$
\sup_t \Biggl[P\Bigl( \max_j |Z_j| \leq t+\epsilon\Bigr) - P\Bigl( \max_j |Z_j| \leq t\Bigr)\Biggr] \leq
c  \epsilon \sqrt{ \log(k/\epsilon)}
$$
and
$$
\sup_t [P(\max_j Z_j \leq t + \epsilon) - P(\max_j Z_j \leq t)] \leq c  \epsilon\sqrt{ \log(k/\epsilon)}.
$$

\end{lemma}

\begin{proof}
Let $Y = \max_j Z_j$.
Then
\begin{align*}
P\Bigl( \max_j Z_j \leq t+\epsilon\Bigr) - P\Bigl( \max_j Z_j \leq t\Bigr) & \leq
P(t-\epsilon \leq Y \leq t+\epsilon)\\
&=
P(-\epsilon \leq Y-t \leq \epsilon)\\
& \leq
P(|Y-t| \leq \epsilon)\\
& \leq
2 \sup_z P(|Y-z| \leq \epsilon) \leq c  \epsilon \sqrt{ \log(k/\epsilon)}
\end{align*}
where the last inequality is precisely the previous
anti-concentration inequality.
\end{proof}

\begin{remark}
A union bound would have
given a bound of order $k\epsilon$ instead
of $\epsilon\sqrt{\log k/\epsilon}$.
Lemma \ref{lemma::anticoncentration}
leads to much sharper bounds in our delta method and bootstrap bounds.
\end{remark}

\begin{theorem}[Gaussian Comparison; Chernozhukov, Chetverikov and Kato 2013]
\label{theorem::comparison}
Let 
$X=(X_1,\ldots, X_k)\sim N(0,\Sigma_X)$ and
$Y=(Y_1,\ldots, Y_k)\sim N(0,\Sigma_Y)$.
Let
$\Delta = \max_{j,k}|\Sigma_X(j,k)-\Sigma_Y(j,k)$.
Then
\begin{equation}
\sup_z\Bigl| P(\max_j X_j \leq z) - P(\max_j Y_j \leq z)\Bigr| \leq C
\Delta^{1/3} (1 \vee \log(k/\Delta))^{2/3}
\end{equation}
where $C$ is only a function of
$\max_j \Sigma_Y(j,j)$ and
$\min_j \Sigma_Y(j,j)$.
\end{theorem}

\subsection{\bf Berry-Esseen Bounds for High-Dimensional Delta Method}

Define
\begin{equation}
B = \Bigl\{ a:\ ||a-\sigma|| \leq C\sqrt{D^2 \log n/n}\Bigr\}.
\end{equation}
It follows from Lemma \ref{lemma::ballbound} that, for large enough $C$,
$P(s\notin B)\leq 1/n^2$.
We assume throughout the analysis that $s\in B$ as the error this incurs
is of smaller order than the rest of the error terms.
Let $\Theta$ 
and $\hat\Theta$ 
be the matrix of partial correlations and the matrix of
estimate partial correlations.
Let $\theta = {\rm vec}(\Theta)$ and 
$\hat\theta = {\rm vec}(\hat\Theta)$.
Recall that
$$
\theta = (\theta_1,\ldots, \theta_{D^2})^T = G(\sigma) =
(g_1(\sigma),\ldots, g_{D^2}(\sigma)).
$$
By Taylor expansion
and (\ref{eq::qsigma}),
\begin{equation}\label{eq::taylor}
\sqrt{n}(\hat\theta-\theta) = \sqrt{n} L (s-\sigma) + n^{-1/2} R =
\sqrt{n} L \overline{V} -\sqrt{n} L q + \frac{1}{\sqrt{n}} R
\end{equation}
where 
$L = d {\rm vec}(G)/d\sigma^T$ so that
$L$ is the $D^2\times D^2$ matrix
whose $j^{\rm th}$ row
is $\ell_j \equiv d g_j(\sigma)/d\sigma^T$.
Similarly,
$R=(R_1,\ldots, R_{D^2})^T$ where
$R_j = \frac{1}{2} \delta^T H_j \delta$
and
$H_j$ is the Hessian of $g_j$,
evaluated at some point between $s$ and $\sigma$.
Let
\begin{equation}\label{eq::Gamma}
\Gamma = {\rm Var}(\sqrt{n}L (s-\sigma)) = L T_n L^T\ \ \ {\rm and}\ \ \ 
\gamma = {\rm diag}(\Gamma).
\end{equation}
Let 
$$
Z = \sqrt{n}\gamma^{-1/2} (\hat\theta - \theta) = (Z_1,\ldots, Z_{D^2})^T
$$
where
$Z_j = \sqrt{n} (\hat\theta_j - \theta_j)/e_j$
is the normalized estimate
and
$e_j = \gamma^{1/2}(j,j) = \sqrt{\ell_j(\sigma)^T T(\sigma)\ell_j(\sigma)}$.
The covariance of $Z$ is
$$
\tilde\Gamma = \gamma^{-1/2}\Gamma \gamma^{-1/2}.
$$
Note that
$\tilde \Gamma_{jj} =1$ for all $j$.

\begin{theorem}\label{theorem::delta}
Let $W\sim N(0,\tilde\Gamma)$ where
$W\in \mathbb{R}^{D^2}$
and let
$$
\gamma_n = \max_j \sup_{a\in B}  \frac{|\!|\!|H_j(a)|\!|\!|}{\sqrt{\ell_j(a)^T T_n(a)\ell_j(a)}}
\ \ \ {\rm and}\ \ \ 
\xi_n = \max_j \sup_{a\in B}||\gamma^{-1/2} \ell_j(a)||_1.
$$
Then, 
\begin{equation}
\sup_z \Bigl|P( \max_j |Z_j| \leq z) - P( \max_j |W_j| \leq z)\Bigr| \preceq A_n
\end{equation}
where
\begin{equation}\label{eq::An}
A_n =
\frac{\log D}{n^{1/8}} + 
\frac{4(\gamma_n+\xi_n)}{\zeta^2} \sqrt{ \frac{ \log (Dn)}{n}}
\sqrt{ \log\left( \frac{D\zeta^2}{4(\gamma_n+\xi_n)}\sqrt{\frac{n}{\log (Dn)}}\right)}.
\end{equation}
Hence, if
$z_\alpha \equiv -\Phi^{-1}(\alpha/D^2)$ then
$$
P( \max_j |Z_j| > z_\alpha) \leq \alpha + A_n.
$$
\end{theorem}

\begin{remark}
  In the above result, the dimension enters mainly through the terms
  $\gamma_n$ and $\xi_n$. Except for these terms, the dependence on
  $D$ is only logarithmic.  We discuss these terms in Section
  \ref{section::error}.
\end{remark}

\begin{proof}
  By (\ref{eq::taylor}),
$$
Z = \sqrt{n}\gamma^{-1/2} (\hat\theta-\theta)=
\sqrt{n} \gamma^{-1/2} L\overline{V} - \sqrt{n}\gamma^{-1/2} L q + \frac{1}{\sqrt{n}}\gamma^{-1/2} R.
$$
Note that ${\rm Var}(W_i) = {\rm Var}(\sqrt{n} \gamma^{-1/2}L
\overline{V})$.  Fix $\epsilon>0$ and let
$$
E = \Biggl\{ \Bigl| \Bigl| \frac{\gamma^{-1/2} R}{\sqrt{n}} \Bigr| \Bigr|_\infty \leq \epsilon \Biggr\}
\ \ \ {\rm and}\ \ \ 
E' = \Biggl\{ || \sqrt{n}\gamma^{-1/2}L q ||_\infty \leq \epsilon \Biggr\}.
$$
Now
\begin{align*}
P( \max_j |Z_j| \leq z) &=
P( ||\sqrt{n} \gamma^{-1/2} L\overline{V} - \sqrt{n}\gamma^{-1/2} L q
+ \frac{1}{\sqrt{n}}\gamma^{-1/2} R||_\infty \leq z)\\
& \leq
P( ||\sqrt{n} \gamma^{-1/2} L\overline{V}||_\infty 
- ||\sqrt{n}\gamma^{-1/2} L q||_\infty 
- ||\frac{1}{\sqrt{n}}\gamma^{-1/2} R||_\infty \leq z)\\
&=
P( ||\sqrt{n} \gamma^{-1/2} L\overline{V}||_\infty 
- ||\sqrt{n}\gamma^{-1/2} L q||_\infty 
- ||\frac{1}{\sqrt{n}}\gamma^{-1/2} R||_\infty \leq z,\, E)\\
& \ \ \ +
P( ||\sqrt{n} \gamma^{-1/2} L\overline{V}||_\infty 
- ||\sqrt{n}\gamma^{-1/2} L q||_\infty 
- ||\frac{1}{\sqrt{n}}\gamma^{-1/2} R||_\infty \leq z,\ E^c)\\
& \leq
P( ||\sqrt{n} \gamma^{-1/2} L\overline{V}||_\infty 
- ||\sqrt{n}\gamma^{-1/2} L q||_\infty  \leq z+\epsilon) + P(E^c)\\
&=
P( ||\sqrt{n} \gamma^{-1/2} L\overline{V}||_\infty 
- ||\sqrt{n}\gamma^{-1/2} L q||_\infty  \leq z+\epsilon, E')\\
& \ \ \  +
P( ||\sqrt{n} \gamma^{-1/2} L\overline{V}||_\infty 
- ||\sqrt{n}\gamma^{-1/2} L q||_\infty  \leq z+\epsilon, (E')^c) +  P(E^c)\\
& \leq
P( ||\sqrt{n} \gamma^{-1/2} L\overline{V}||_\infty 
   \leq z+2\epsilon) + P(E^c) + P((E')^c).
\end{align*}
So,
\begin{align*}
P( \max_j |Z_j| &\leq z) - P( \max_j |W_j| \leq z)\\
& \leq
P\left(\Biggl|\Biggl| \sqrt{n}\gamma^{-1/2} L\overline{V}
   \Biggr| \Biggr|_\infty \leq z+2\epsilon\right) - 
P( \max_j |W_j| \leq z+2\epsilon)\\
&\ \ \  + P( \max_j |W_j| \leq z+2\epsilon) 
- P( \max_j |W_j| \leq z) + P(E^c) + P((E')^c)\\
& \leq
P\left(\Biggl|\Biggl| \sqrt{n} \gamma^{-1/2}L\overline{V}   
\Biggr| \Biggr|_\infty \leq z+2\epsilon\right) - 
P( \max_j |W_j| \leq z+2\epsilon) + C \epsilon\sqrt{\log D/\epsilon} + P(E^c)+ P((E')^c)\\
& \leq
C \frac{\log D}{n^{1/8}} + C \epsilon \sqrt{\log D/\epsilon} + P(E^c)+ P((E')^c)
\end{align*}
where we used Theorem \ref{theorem::clt} applied to $\overline{V}^* =
\gamma^{-1/2}L\overline{V}$ and Lemma \ref{lemma::anticoncentration}.
Recall that $s\in B$ except on a set of probability $1/n^2$ and on
this set,
$$
\left(\frac{\gamma^{-1/2} R}{\sqrt{n}}\right)_j =
\frac{\delta^T H_j \delta}{\sqrt{n\ell_j^T T_n\ell_j}}  \leq
\gamma_n \sqrt{n}||s-\sigma||_\infty^2
$$
and so by Lemma \ref{lemma::ballbound},
$$
P(E^c) \leq 2D^2 \exp\left( - \frac{-\sqrt{n} \zeta^2 \epsilon^2}{2\gamma_n^2}\right).
$$
Choosing
$$
\epsilon =  \frac{4(\gamma_n+\xi_n)}{\zeta^2} \sqrt{ \frac{ \log (Dn)}{n}}
$$
we have
$P(E^c)  \leq \frac{1}{n^2}$
and
$$
\epsilon\sqrt{\log D/\epsilon} \leq
\frac{4(\gamma_n+\xi_n)}{\zeta^2} \sqrt{ \frac{ \log (Dn)}{n}}
\sqrt{ \log\left( \frac{D\zeta^2}{4(\gamma_n+\xi_n)}\sqrt{\frac{n}{\log (Dn)}}\right)}.
$$
Using Holder's inequality,
$$
|\gamma^{-1/2}\ell_j^T q| \leq ||q||_\infty\ ||\gamma^{-1/2}\ell_j||_1 \leq ||q||_\infty\ \xi_n
$$
so that
$||\gamma^{-1/2} L q||_\infty \leq ||q||_\infty\ \xi_n$.
Hence, using Lemma (\ref{lemma::ballbound}),
\begin{align*}
P((E')^c) &\leq P( ||q||_\infty > \epsilon/\sqrt{\xi_n n}) \leq
4D^2 e^{-\sqrt{n}\zeta^2 \epsilon/(2\xi_n)} \leq \frac{1}{n^2}
\end{align*}
The result follows by computing a similar lower bound and taking the
supremum over $z$.  For the last statement, note that $W_j\sim
N(0,1)$.  So
\begin{align*}
P( \max_j |Z_j| > z_\alpha) &\leq
P( \max_j |W_j| > z_\alpha) +A_n \leq
\sum_j P( |W_j| > z_\alpha) +A_n \leq \alpha + A_n.
\end{align*}
\end{proof}

In practice we need to use
$T_j = \sqrt{n}(\hat\theta_j - \theta_j)/\hat e_j$
where
$\hat e_j = \sqrt{\ell_j(s)^T T(s)\ell_j(s)}\equiv U_j(s)$
is the estimated standard error.
We have the following result for this case.

\begin{theorem}
Define $\gamma_n$ and $\xi_n$
as in the previous theorem.
Let
$$
\rho_n = \max_j \sup_{a\in B}  \frac{||U'_j(a)||_1}{\sqrt{\ell_j(a)^T T_n(a)\ell_j(a)}}
$$
where
$U_j(a) = \sqrt{\ell_j^T(a) T(a) \ell_j(a)}$.
Then, 
$$
\sup_z |P( \max_j |T_j| \leq z) - P( \max_j |W_j| \leq z)| \preceq
A_n + \rho_n \sqrt{\frac{ \log n}{n}}
$$
where $A_n$ is defined in (\ref{eq::An}).
If $z \equiv -\Phi^{-1}(\alpha/D^2)$ then
$$
\sup_z |P( \max_j |T_j| > z)| \leq \alpha +
A_n + \rho_n \sqrt{\frac{ \log n}{n}}.
$$
\end{theorem}

\begin{proof}
Let $E=\{ \max_j e_{j}/\hat{e}_{j} < 1+\epsilon\}$ and 
$F =\{ \max Z_{j} < u/\epsilon\}$ where 
$\epsilon = (4\rho_n/\zeta)\sqrt{ \log n/(n\zeta^2)}$ and $u = \epsilon \sqrt{\log(n)}$.  
Note that
$e_{j} - \hat{e}_{j} = U_j(\sigma) - U_j(s) = (\sigma-s)^T U'_j$
where $U'$ is the gradient of $U$ evaluated at some point between
$s$ and $\sigma$.  Then, for $0 < \epsilon \leq 1$,
\begin{align*}
P(E^c) &\leq  P\left( \max_j \frac{e_{j}-\hat{e}_{j}}{e_{j}} > \frac{\epsilon}{1+\epsilon}\right)=
P\left( \max_j \frac{U_j(\sigma)-U_j(s)}{e_{j}} > \frac{\epsilon}{1+\epsilon}\right)\\
&= P\left(\max_j \frac{ (\sigma-s)^T U'_j }{e_{j}} > \frac{\epsilon}{1+\epsilon}\right)\leq
P\left( \frac{||s-\sigma||_\infty \max_j ||U'_j||_1}{e_{j}}> \frac{\epsilon}{1+\epsilon}\right)\\
& \leq P\left( ||s-\sigma||_\infty \rho_n > \frac{\epsilon}{1+\epsilon}\right)=
P\left( ||s-\sigma||_\infty > \frac{\epsilon}{2\rho_n}\right)\\
& \leq D^2 e^{-n\epsilon^2/(2\rho_n^2)} \leq \frac{1}{n^2}.
\end{align*}
Now,
\begin{align*}
P\left( \max_j \frac{\sqrt{n}(\hat\theta_{j}-\theta_{j})}{\hat e_{j}} \leq z \right) &-
P(\max W_j \leq z)\\
& \hspace{-1in} =P\left(\max_j Z_{j}\left(\frac{e_{j}}{\hat e_{j}}\right) \leq z \right) - P(\max W_j \leq z)\\
&  \hspace{-1in}\leq P\left(\max_j Z_{j} (1-\epsilon) \leq z \right) +P(E^c) - P(\max W_j \leq z)\\
&  \hspace{-1in}=P\left(\max_j Z_{j} -Z_{j}\epsilon \leq z \right) +P(E^c) - P(\max W_j \leq z)\\
& \hspace{-1in}\leq P\left(\max_j Z_{j}  \leq z+u \right) + P(F^c) + P(E^c) - P(\max W_j \leq z)\\
& \hspace{-1in}\leq P\left(\max_j Z_{j}  \leq z+u \right) -P(\max W_j \leq z+u)\\
&\ \ \ + C u \sqrt{\log D/u} + P(F^c) + P(E^c)\\
& \hspace{-1in} \leq
\sup_z \Biggl[P\left(\max_j Z_{j}  \leq z \right) -P(\max_j W_j \leq z)\Biggr]\\
&\ \ \ + C u \sqrt{\log D/u} + P(F^c) + P(E^c)\\
&  \hspace{-1in}\leq
A_n + C u \sqrt{\log D/u} + P(F^c) + P(E^c)
\end{align*}
where $A_n$ is defined in (\ref{eq::An}).
Next,
\begin{align*}
P(F^c) &=P(\max_j Z_{j}> u/\epsilon) \leq P(\max_j W_j > u/\epsilon) + A_n\\
&= P(\max_j W_j > \sqrt{\log n}) + A_n\\
& \leq \frac{E(\max_j W_j)}{\sqrt{\log n}} + A_n \preceq
\frac{\sqrt{\log D}}{\sqrt{\log n}} + A_n \preceq A_n.
\end{align*}
So
\begin{align*}
\sup_z |P( \max_j |T_j| \leq z) &- P( \max_j |W_j| \leq z)|  \preceq
\sup_z
[P\left(\max_j Z_{j}  \leq z \right) -P(\max W_j \leq z) ]  + A_n 
\\ &\ \ \ \ +
\frac{1}{n^2} + C u \sqrt{\log D}\\
& \preceq
A_n + \rho_n \sqrt{\frac{ \log n}{n}}.
\end{align*}
A similar lower bound completes the proof.
\end{proof}

\subsection{\bf The Bootstrap}

In this section we assume that $\max_j| Y(j)| \leq B$ for some $B<
\infty$.  This is not necessary but it simplifies the proofs.  We do
not require that $B$ be known.  Let $Y_1^*,\ldots, Y_n^*$ be a sample
from the empirical distribution and let $s^*$ be the corresponding
(vectorized) sample covariance.  Now let $\hat\theta^*$ be the partial
correlations computed from $Y_1^*,\ldots, Y_n^*\sim P_n$ where $P_n$
is the empirical distribution.  The (un-normalized) bootstrap
rectangle for $\theta$ is
$$
{\cal R}_n=\Biggl\{ \theta: ||\theta-\hat\theta||_\infty \leq \frac{Z_\alpha}{\sqrt{n}}\Biggr\}
$$
where $Z_\alpha = \hat F^{-1}(1-\alpha)$
and
\begin{equation}
\hat F(z) = P\biggl(\sqrt{n}||\hat\theta^*-\hat\theta||_\infty \leq z \biggm| Y_1,\ldots,Y_n\biggr)
\end{equation}
is the bootstrap approximation to
$$
F(z) = P(\sqrt{n}||\hat\theta-\theta||_\infty \leq z).
$$
The accuracy of the coverage of the bootstrap rectangle depends on
$\sup_z| \hat F(z) - F(z)|$.

Let
$$
\Gamma = {\rm Var}(\sqrt{n}L (s-\sigma)) = L T_n L^T.
$$
Let $Z\sim N(0,\Gamma)$ where
$Z\in \mathbb{R}^{D^2}$.
First we need the following limit theorem for the un-normalized statistics.

\begin{theorem}
Define $\gamma_n' = \max_j \sup_{a\in B} |\!|\!|H_j(a)|\!|\!|$
and
$\xi_n' = \max_j \sup_{a\in B} ||\ell_j(a)||_1$.
Then
$$
\sup_z 
\Biggl|P(\sqrt{n}||\hat\theta-\theta||_\infty \leq z) - P(||Z||_\infty \leq z) \Biggr| \preceq
\frac{\log D}{n^{1/8}} + A_n'
$$
where
\begin{equation}\label{eq::Anp}
A_n' =
\frac{\log D}{n^{1/8}} + 
\frac{4(\gamma_n'+\xi_n')}{\zeta^2} \sqrt{ \frac{ \log (Dn)}{n}}
\sqrt{ \log\left( \frac{D\zeta^2}{4(\gamma_n'+\xi_n')}\sqrt{\frac{n}{\log (Dn)}}\right)}.
\end{equation}
\end{theorem}

\begin{proof}
The proof is the same as the proof of Theorem
\ref{theorem::delta}
with $\gamma_n'$ and $\xi_n'$ replacing
$\gamma_n$ and $\xi_n$.
\end{proof}

Now we bound 
$\sup_z| \hat F(z) - F(z)|$.

\begin{theorem}\label{thm::bootstrap}
$$
\sup_z |\hat F(z) - F(z)| \preceq
\frac{ \log D}{n^{1/8}} +  (\gamma_n'+\xi_n') \sqrt{\log n /n} +
O_P\left(\left(\frac{\log D}{n}\right)^{1/6}\right)
$$
and hence
$$
P(\theta \notin {\cal R}) \leq \alpha +
\frac{ \log D}{n^{1/8}} +  (\gamma_n'+\xi_n') \sqrt{\log n /n} +
O\left(\left(\frac{\log D}{n}\right)^{1/6}\right).
$$
\end{theorem}

\begin{proof}
Let $Z\sim N(0,\Gamma)$ and let
$Z'\sim N(0,\Gamma_n)$
where
$\Gamma_n = {\rm Var}(\sqrt{n}L(s^*-s)|Y_1,\ldots, Y_n)$.
Then
\begin{align*}
\sup_z |\hat F(z) - F(z)| &\leq
\sup_z \Bigl|F(z) - P(||Z||_\infty \leq z)\Bigr| +
\sup_z \Bigl|\hat F(z) - P(||Z'||_\infty \leq z)\Bigr|\\
&\  + \sup_z \Bigl| P(||Z'||_\infty \leq z) - P(||Z||_\infty \leq z)\Bigr|\\
&= I \ \ \ +\ \ \ II\ \ \ +\ \ \ III.
\end{align*}
In the previous theorem, we showed that
$I \leq \frac{\log D}{n^{1/8}} + A_n'.$
For $II$,
we proceed exactly as in the proof for of the previous theorem
but with $P_n$ replacing $P$
(and with $Y_1,\ldots, Y_n$ fixed).
This yields, for any $\epsilon >0$,
\begin{align*}
\hat F(z) - P(||Z'||_\infty \leq z) & \preceq
\frac{\log D}{n^{1/8}} +
\epsilon\sqrt{\log D/\epsilon}\\
&\ \ \ +\ 
P(\sqrt{n}||L q^*||_\infty > \epsilon|Y_1,\ldots, Y_n) +
P(n^{-1/2}||R^*||_\infty > \epsilon| Y_1,\ldots, Y_n)
\end{align*}
where
$q^* = {\rm vec}( (\overline{Y}^* - \overline{Y})(\overline{Y}^* - \overline{Y})^T)$,
$R^*_j = (1/2)\delta^T H^*_j \delta^*$,
$\delta^* = \sqrt{n}(s^*-s)$
and $H_j^*$ is the Hessian of $g_j$
evaluated at a point between $s$ and $s^*$.

Since all the $Y_i$'s are contained in the bounded rectangle
$B\times \cdots \times B$,
it follows that under the empirical measure $P_n$,
$Y_i^*$ is sub-Gaussian with
$\zeta = B$.
It then follows that
$s^*\in B$ expect on a set of probability at most $1/n$.
Choosing
$$
\epsilon =  \frac{4(\gamma_n'+\xi_n')}{B^2} \sqrt{ \frac{ \log (Dn)}{n}}
$$
and arguing as in
the proof of Theorem \ref{theorem::delta}
we conclude that
\begin{align*}
\hat F(z) - P(||Z'||_\infty \leq z) & \preceq
\frac{\log D}{n^{1/8}} +
\epsilon\sqrt{\log D/\epsilon}\\
&\ \ \ +\ 
P(\sqrt{n}||L q^*||_\infty > \epsilon|Y_1,\ldots, Y_n) +
P(n^{-1/2}||R^*||_\infty > \epsilon| Y_1,\ldots, Y_n)\\
& \leq
\frac{\log D}{n^{1/8}} +
O_P(A_n').
\end{align*}
For $III$, we use Theorem \ref{theorem::comparison} which implies that
$$
III \leq
C\, \Delta^{1/3} (1 \vee \log(k/\Delta))^{2/3}
$$
where $\Delta = \max_{s,t}| \Gamma(s,t) - \Gamma_n(s,t)|.$ Each
element of $\Gamma_n(s,t)$ is a sample moment and $\Gamma(s,t)$ is
corresponding population moment, and so, since $P_n$ is sub-Gaussian,
$\Delta = O_P(\sqrt{\log D /n})$.  Hence, $III = O_P\left( \frac{\log
    D}{n}\right)^{1/6}.$
\end{proof}

\subsection{\bf A Super-Accurate Bootstrap}
\label{sec::super}

Now we describe a modified approach to the bootstrap that has coverage
error only $O(\log D/n^{1/8})$ which is much more accurate than the
usual bootstrap as described in the last section.  The idea is very
simple.  Let ${\cal R}$ be the $1-\alpha$ bootstrap confidence
rectangle for $\sigma$ described in Section
\ref{section::correlation-graphs}.  Write $\theta = G(\sigma)$ and
define
$$
{\cal T} = \Bigl\{ G(\sigma):\ \sigma\in {\cal R} \Bigr\}.
$$
By construction, ${\cal T}$ inherits the
coverage properties of ${\cal R}$ and so we have immediately:

\begin{corollary}
$$
P(\theta\in{\cal T}) \geq 1-\alpha - O\left(\frac{ \log D}{n^{1/8}}\right) - O\left(\frac{ \log D}{n}\right)^{1/6}.
$$
\end{corollary}

The set ${\cal T}$ then defines confidence sets for each $\theta_j$, namely,
$$
C_j = \Bigl[ \inf \{g_j(\sigma):\ \sigma\in {\cal R}\},\ \ \ \sup \{g_j(\sigma):\ \sigma\in {\cal R}\}\Bigr].
$$
We should stress that, in general, obtaining a confidence set by
mapping a confidence rectangle can lead to wide intervals.  However,
our foremost concern in this paper is coverage accuracy.

Constructing the set ${\cal T}$ can be difficult.  But it is easy to
get an approximation.  We draw a large sample $\sigma_1,\ldots,
\sigma_N$ from a uniform distribution on the rectangle ${\cal R}$.
Now let
$$
\underline{\theta}_j = \min_{1\leq s\leq N}g_j(\sigma_s),\ \ \ 
\overline{\theta}_j = \max_{1\leq s\leq N}g_j(\sigma_s).
$$
Then $[\underline{\theta}_j,\overline{\theta}_j]$ approximates the
confidence interval for $\theta_j$.  Alternatively, we take
$\sigma_1,\ldots, \sigma_N$ to be the bootstrap replications that are
contained in ${\cal R}$.  Note that there is no need for a multiple
comparison correction as the original confidence rectangle is a
simultaneous confidence set.

\subsection{\bf Comments on the Error Terms}
\label{section::error}

The accuracy of the delta method depends on the dimension $D$ mainly
through the terms $\gamma_n$, $\xi_n$ and $\rho_n$.  Similarly, the
accuracy of the (first version of the) bootstrap depends on $\gamma_n'$ and $\xi_n'$.  In
this section we look at the size of these terms.  We focus on
$\gamma_n'$ and $\xi_n'$.

Recall that $\ell_j = d \theta_j/d\sigma^T$.
Then
$$
\ell_{j}(\sigma) = \frac{d \theta_{j}}{d\sigma^T} =
\frac{d \theta_{j}}{d\omega^T}
\frac{d \omega}{d\sigma^T}.
$$
Let $(s,t)$ be such that
$\theta_j = \Theta_{st}$.
Then,
$\frac{d \theta_{j}}{d\omega^T}$ is $1\times D^2$ and
$\frac{d \omega}{d\sigma^T}$ is $D^2 \times D^2$.
Now
$\frac{d \omega}{d\sigma^T} = - \Omega \otimes \Omega$
and
$\frac{d \theta_{j}}{d\omega^T}$ is 0
except for three entries, namely,
$$
\frac{d \theta_{j}}{d \Omega_{ss}} = - \frac{\theta_{j}}{2\Omega_{ss}},\ \ \ \ \ 
\frac{d \theta_{j}}{d \Omega_{tt}} = - \frac{\theta_{j}}{2\Omega_{ss}},\ \ \ \ \ 
\frac{d \theta_{j}}{d \Omega_{st}} =  \frac{\theta_{j}}{\Omega_{st}}.
$$
Define
$(J,K,M)$ by
$\sigma_{J} = \Sigma_{ss}$,
$\sigma_{K} = \Sigma_{tt}$ and
$\sigma_{M} = \Sigma_{st}$.
Then
\begin{equation}
\ell_j = \frac{d \theta_{j}}{d\sigma^T} =
\frac{\theta_{j}}{2\Omega_{ss}} [\Omega \otimes \Omega]_J +
\frac{\theta_{j}}{\Omega_{st}} [\Omega \otimes \Omega]_M + 
\frac{\theta_{j}}{2\Omega_{tt}} [\Omega \otimes \Omega]_{K}=f_j (\Omega\otimes\Omega)
\end{equation}
where
$[A]_j$ denotes the $j^{\rm th}$ row of $A$
and $f_j$
is a sparse vector that is 0
except for three entries.

Now the Hessian is
$H_j= \left(\frac{d \ell_1}{d\sigma^T},\ldots, \frac{d \ell_{D^2}}{d\sigma^T}\right)^T$
where
$$
\frac{d \ell_j}{d\sigma^T} = 
\frac{d \ell_j}{d\omega^T}\frac{d \omega}{d\sigma^T} = - \frac{d \ell_j}{d\omega^T}(\Omega\otimes\Omega).
$$
Now
\begin{align*}
\frac{d \ell_j}{d\omega^T} &=
\left( \left(\frac{d\omega}{d\sigma^T}\right)^T \otimes I\right)
\frac{d}{d\omega^T} \left(\frac{d\theta_j}{d\omega}\right) +
\left( I \otimes \frac{d\theta_j}{d\omega^T}\right)
\frac{d}{d\omega^T} \frac{d\omega}{d\sigma^T}\\
&= -
(\Omega\otimes\Omega\otimes I)f_j -
(I\otimes f_j) \frac{d}{d\omega^T} (\Omega\otimes\Omega)\\
&=-
(\Omega\otimes\Omega\otimes I)f_j -
(I\otimes f_j) (I_D \otimes K_{(D,D)} \otimes I_D)
(I_{D^2} \otimes {\rm vec}(\Omega):{\rm vec}(\Omega)\otimes I_{D^2});
\end{align*}
where we used the fact that
$$
\frac{d {\rm vec} (\Omega \otimes \Omega)}{d\omega^T} =
(I_D \otimes K_{(D,D)} \otimes I_D)
(I_{D^2} \otimes {\rm vec}(\Omega):{\rm vec}(\Omega)\otimes I_{D^2});
$$
see, for example, p 185 of
\cite{magnus1988matrix}
Note that $||f_j||_0 = O(1)$ independent of $D$.
The presence of this sparse vector
helps to prevent the gradient and Hessian from getting too large.

By direct examination of $\ell_j$ and $H_j$ we see that the size of
$\gamma_n'$ and $\xi_n'$ depends on how dense $\Omega$ is.  In
particular, when $\Omega$ is diagonally dominant, $\gamma_n'$ and
$\xi_n'$ are both $O(1)$.  In this case the error terms have size
$O((\log D_n)/n^{1/8})$.  However, if $\Omega$ is dense, then
$||\ell_j||_1$ can be of order $O(D^2)$ and and $|\!|\!| H_j|\!|\!|$
can be of order $O(D^4)$.  In this case the error can be as large as
$D^4/n^{1/8}$.  On the other hand, the bootstrap in Section
\ref{sec::super} always has accuracy $O((\log D_n)/n^{1/8})$.  But the
length of the intervals could be large when $\Omega$ is dense.  And
note that even in the favorable case, we still require $D_n < n$ for
the results to hold.  (We conjecture that this can be relaxed by using
shrinkage methods as in \cite{schafer2005shrinkage}.)  These
observations motivate the methods in the next section which avoid
direct inferences about the partial correlation graph in the
high-dimensional case.

It is interesting to compare the size of the errors to other work on
inference with increasing dimension.  For example,
\cite{portnoy1988asymptotic} gets accuracy $\sqrt{D^{3/2}/n}$ for
maximum likelihood estimators in exponential families and
\cite{mammen1993bootstrap} gets accuracy $\sqrt{D^2/n}$ for the
bootstrap for linear models.

\subsection{\bf Back To Graphs}

Finally, we can use the above methods for estimating a graph with
confidence guarantees.  We put an edge between $j$ and $k$ only if 0
is excluded from the confidence interval for $\theta_{jk}$.  The
desired guarantee stated in (\ref{eq::goal1}) then holds.

\section{The High Dimensional Case}
\label{section::methods3}

Now we consider the case where $D_n > n$.  We present three methods
for dealing with the high-dimensional case:

\begin{enumerate}

\item {\bf Correlation graphs}.  This is a common technique in
  biostatistics.  We connect two nodes if the confidence interval for
  two variables excludes $[-\epsilon,\epsilon]$ for some threshold
  $\epsilon\in [0,1]$.  Our contribution here is to provide confidence
  guarantees using the bootstrap that are valid as long as $D =
  o(e^{n^{1/7}})$.  In this paper we use $\epsilon=0$.

\item {\bf Cluster graphs}. We cluster the features and average the
  features within each cluster.  As long as the number of clusters $L$
  is $o(n)$ we get valid inferences.  Related to cluster graphs are
  {\bf block graphs}.  In this case, we again cluster the nodes.  But
  then we make no connections between clusters and we use an
  undirected graph within clusters.

\item {\bf Restricted Graphs}. Define the restricted partial correlation
$$
\theta_{jk} \equiv \sup_{|S| \leq L} |\theta(Y_j,Y_k|Y_S)|
$$
where $L$ is some fixed number, $\theta(Y_j,Y_k|Y_S)$ is the partial
correlation between $Y_j$ and $Y_k$ given the set of variables $Y_S$
where $S$ varies over all subsets of $\{1,\ldots,D\}-\{j,k\}$ of size
$L$ These are sometimes called lower-order partial correlations.  Now
construct a graph based on the restricted partial correlations.  Note
that $L=0$ is a correlation graph and $L=D$ is a partial correlation
graph.  (This is similar to the idea in Castelo and Roverato, 2006).
The bootstrap leads to valid inferences only requiring $D =
o(e^{n^{1/7}})$.

\end{enumerate}

\begin{remark}
  Following \cite{schafer2005shrinkage}, we could estimate $U =
  (1-\lambda)\Sigma + \lambda T$ where $T$ is, for example, a diagonal
  matrix.  The graph is constructed from biased partial correlations
  corresponding to $U^{-1}$.  When $\lambda$ is close to 1,
  high-dimensional asymptotic confidence intervals have accurate
  coverage.  Thus we have a bias-validity tradeoff.  Investigating
  this tradeoff is quite involved and so we will examine this method
  elsewhere.
\end{remark}

In this section we make the following assumptions.

(A1) $Y$ and ${\rm vec}(Y Y^T)$ are sub-Gaussian.

(A2) $\max_j \mathbb{E} |V_i(j)|^3 \leq C$
where
$V_i = {\rm vec}[(Y_i-\mu)(Y_i-\mu)^T] - \sigma$.

(A3) $D_n = o(e^{n^{1/7}})$.

\vspace{1cm}

The proofs of the results in this section are similar to those in
Section \ref{section::moderate} but they are easier as the error terms
are, by design, not dependent on dimension sensitive quantities like
$\gamma_n$ and $\xi_n$.  Because of this, we shall only present proof
outlines.

\subsection{\bf Correlation graphs}
\label{section::correlation-graphs}

The simplest approach to constructing graphs is to use correlation or
covariances rather than partial correlation.  Let $\rho_{jk}$ denoted
the correlation between $Y(j)$ and $Y(k)$.  The true graph
$G_\epsilon$ connects $j$ and $k$ if $|\rho(j,k)| > \epsilon$ where $0
\leq \epsilon\leq 1$ is some user-specified threshold.  The algorithm
is in Figure \ref{fig::correlation-algorithm}.  Of course, we can use
either $\rho$ or $\sigma$; we get the same graph from either.

\begin{figure}
\fbox{\parbox{7in}{
\begin{enum}
\item Select a threshold $\epsilon$.
\item Compute the sample covariance matrix $R$.
\item Construct a $1-\alpha$ bootstrap confidence rectangle ${\cal R}$
for the correlations.
\item Put an edge between nodes $j$ and $k$ if $[-\epsilon,\epsilon]$
is not in the confidence interval for $\rho_{jk}$.
\end{enum}}}
\caption{The Correlation Graph Algorithm.}
\label{fig::correlation-algorithm}
\end{figure}

\begin{theorem}\label{thm::correlations}
  Let $r_{jk}$ denote the sample correlation between $Y(j)$ and $Y(k)$
  and let $r$ be the $D^2\times 1$ vector of correlations.  Similarly,
  let $\rho$ be the vector of true correlations.  Define $Z_\alpha$ by
  the bootstrap equation
\begin{equation}
P\Bigl(\max_{jk} \sqrt{n}|r_{jk}^* - r_{jk}| > Z_\alpha \ \Bigm| \ Y_1,\ldots, Y_n\Bigr) = \alpha.
\end{equation}
Let
$$
{\cal R} = \Bigl\{ a\in \mathbb{R}^{D^2}:\ ||a-r||_\infty \leq \frac{Z_\alpha}{\sqrt{n}} \Bigr\}.
$$
Then
$$
P(\rho \in {\cal R}) \geq 1-\alpha -
O\left(\frac{ \log D}{n^{1/8}}\right)- O\left( \frac{\log D}{n}\right)^{1/6}.
$$
We thus have
\begin{equation}
P( \hat G_\epsilon \subset G_\epsilon\ \ {\rm for\ all\ }\epsilon) \geq 1-\alpha +
\frac{ \log D}{n^{1/8}} +
O\left( \frac{\log D}{n}\right)^{1/6}.
\end{equation}
\end{theorem}

\begin{remark}
A very refined Berry-Esseen result for
a single correlation was obtained by \cite{Pinelis}.
\end{remark}

{\bf Proof Outline.}  The proof is the same as the proof of Theorem
\ref{thm::bootstrap}.  However, in this case, it is easy to see that
$\gamma_n'$ and $\xi_n'$ are $O(1)$, independent of the $D$ since the
gradient $\ell_j$ and Hessian $H_j$ is a function only of the
bivariate distribution of $(Y(s),Y(t))$ corresponding to the
correlation. $\square$

\subsection{\bf Cluster Graphs and Block Graphs}

The idea here is to partition the features into clusters, average the
features within each cluster and then form the graph for the new
derived features.  If the clusters are sufficiently few, then valid
inference is possible.

There are many clustering methods.  Here we consider choosing a set of
representative features --- or prototypes --- using the $L$-centers
algorithm, which we describe below.  Then we assign each feature to
its nearest center.  We average the features within each cluster and
then find the undirected graph of these new $L$ derived features.  Let
$\tilde G$ be the graph for these new features.  We estimate $\tilde
G$ using confidence intervals for the partial correlations.  Note that
the graph $\tilde G$ as well as the estimated graph $\hat G$ are both
random.

To ensure the validity of the confidence intervals, we use data
spitting.  We split the data randomly into two halves.  The first half
is used for clustering.  The confidence intervals are constructed from
the second half of the data.

\begin{figure}
\fbox{\parbox{7in}{
\begin{enum}
\item Choose $L = o(n)$.
\item Randomly split the data into two halves ${\cal D}_1$ and ${\cal D}_2$.
\item Using ${\cal D}_1$ select $L$ proto-features:
\begin{enum}
\item Choose a feature $j$ randomly and set ${\cal S} = \{j\}$ and
${\cal C} = \{1,\ldots,D\} - {\cal S}$.
\item Repeat until ${\cal S}$ has $L$ elements:
\begin{enum} 
\item For each $j\in {\cal C}$ compute the minimum distance
$d_j = \min_{i\in {\cal S}}d(i,j)$.
\item Find $j\in {\cal C}$ to maximize $d_j$.
Move $j$ from ${\cal C}$ to ${\cal S}$.
\end{enum}
\item For $L$ clusters by assigning each feature to its closest center.
\item Average the features within each clusters.
\end{enum}
\item Using ${\cal D}_2$, construct a confidence graph for the $L$ new features using either the delta method
or the bootstrap from Section \ref{section::moderate}.
\item (Optional): Construct a correlation graph for the features within each cluster.
\end{enum}}}
\caption{The Cluster Graph Algorithm}
\label{fig::cluster-algorithm}
\end{figure}

The cluster-graph algorithm is described in Figure
\ref{fig::cluster-algorithm}.  It is assumed in the algorithm that the
number of features $L=o(n)$ is specified by the user.  An improvement
is to use a data-driven approach to choosing $L$.  We leave this to
future work.

The asymptotic validity of the method follows from the results in
Section \ref{section::moderate} together with the data-splitting step.
Without the data-splitting step, the proofs in Section
\ref{section::moderate} would not be valid since the feature selection
process would introduce a bias.  The independence introduced by the
splitting thus seems critical.  Whether it is possible to eliminate
the data-splitting is an open problem.  Let us state, without proof,
the validity assuming the bootstrap is used.  A similar result holds
for the delta method.

\begin{theorem}
  Let $\theta$ be the vector of $k$ partial correlations for the
  features selected from the first half of the data.  Let ${\cal R}$
  be the confidence rectangle using the second half of the data.  Then
\begin{equation}
P(\theta \notin {\cal R}) \leq \alpha +
\frac{ (\log L)}{n^{1/8}} +  (\gamma_n'+\xi_n') \sqrt{\log n /n} +
O\left(\left(\frac{\log L}{n}\right)^{1/6}\right)
\end{equation}
where $\gamma_n'$ and $\xi_n'$ are functions
of the distribution of the selected features.
\end{theorem}

An alternative is to use block graphs.  For block graphs, we first
cluster the nodes.  Then we make no connections between clusters and
we use an undirected graph within clusters based on the bootstrap.  In
this case, it is required that the number of nodes within each block
be $o(n)$.  However, our experiments with block graphs have been
disappointing and we do not pursue block graphs further.

Yet another possibility is as follows.  For each $(j,k)$ let $Z_{jk}$
be a dimension reduction of the variables $(Y(s):\ s\neq j,k)$.  Then
we could estimate the partial correaltion of $Y(j)$ and $Y(k)$ given
$Z_{jk}$.  This would require a separate dimension reduction step for
each pair $(j,k)$.

\subsection{\bf Restricted Partial Correlations}

Instead of building a graph from partial correlations, we can use a
weaker measure of dependence.  Motivated by \cite{castelo2006robust},
we define
\begin{equation}
\theta_{jk} = \sup_{|S|\leq L} |\theta(X_i,X_j|X_S)|.
\end{equation}
For $L=0$ we get a correlation graph.  For $L=D$ we get back the usual
partial correlation graph.  By choosing $L = o(n)$ we get something in
between these two cases while still retaining validity of the
confidence intervals.

The estimate of $\theta_{jk}$ is the sample version
\begin{equation}
\hat\theta_{jk} = \sup_{|S|\leq k} |\hat\theta(X_i,X_j|X_S)|.
\end{equation}

\begin{theorem}
Define $Z_\alpha$ by the bootstrap equation
\begin{equation}
P\Bigl(\max_{jk} \sqrt{n}|\hat\theta^* - \hat\theta| > Z_\alpha \ \Bigm| \ Y_1,\ldots, Y_n\Bigr) = \alpha.
\end{equation}
Let
$$
{\cal R} = \Bigl\{ a\in \mathbb{R}^{D^2}:\ ||a-\hat\theta||_\infty \leq \frac{Z_\alpha}{\sqrt{n}} \Bigr\}.
$$
Then
$$
P(\theta \in {\cal R}) \geq 1-\alpha -
O\left(\frac{ (\log L)}{n^{1/8}}\right)-
O\left(\left(\frac{\log L}{n}\right)^{1/6}\right).
$$
\end{theorem}

The proof is basically the same as the proof of Theorem
\ref{thm::correlations}.  We remark, however, that in this case, $L$
has to be fixed and chosen in advance.

We think that the restricted partial correlation idea is very
promising but currently we have no efficient way to compute the graph
this way.  To compute the restricted partial correlation we would need
to do the following: for each pair $(j,k)$ we have to search over the
$\binom{D-2}{L}$ subsets and find the maximum.  This is repeated for
all $D^2$ pairs.  Then the entire procedure needs to be bootstrapped.
Despite the fact that the method is currently not computationally
feasible, we include it because we believe that it may be possible in
the future to find efficient computational approximations.

\section{Experiments}
\label{section::experiments}

In this section we illustrate the methods with some simple examples.
We consider three models:
\begin{enum}
\item Dense Model: $\Omega_{jk} = a$ for all $j\neq k$.
\item Markov Chain: $X_j = a X_{j+1} + \epsilon_j$.
\item Structural Equation Model: $X_j = a\sum_{s=1}^{j-1} X_s + \epsilon_j$, $j=2,\ldots, D$.
\end{enum}
The purpose of the experiments is to get some intuitive sense of how
much information in the original graph is captured in the dimension
reduced graph.

In each case we show results for bootstrap.  We stopped when the
results became numerically unstable.  Then we increased the dimension
and switched to the high dimensional methods, namely, the cluster
graphs, the correlation graphs and the restricted graphs.  (We do not
include the block graphs which did not work well.)  The results are in
Figures \ref{fig::UndirectedDense}, \ref{fig::UndirectedMarkov},
\ref{fig::UndirectedSEM}, \ref{fig::UndirectedDenseProto},
\ref{fig::UndirectedMarkovProto} and \ref{fig::UndirectedSEMProto}.

The results for the dense model are good up to $D=50$.  After that,
the cluster graph method is used and it clearly captures the
qualitative features of the graph.  or the Markov graph, validity
holds as $D$ increases but the power starts to decrease leading to
missing edges.  The cluster graph is interesting here as it obviously
cannot reconstruct the Markov structure but still does capture
interesting qualitative features of the underlying graph.  The SEM
model is difficult; it is a complete graph but some edges are harder
to detect.  The power again falls off as $D$ increases.  Again we see
that the cluster graph loses information but permits us to find a
graph with qualitative features similar to the true graph with higher
dimensions.

The correlation graph for the dense and SEM models, while preserving
validity has essentially no power.  More precisely, the graphical
model leaves a very small imprint in the correlation matrix.  For
example, the covariance in the dense model is easily seen to be
$O(a/D)$. So while the inverse covariance matrix is dense, the
covariance matrix has small entries.  The correlation graph for the
Markov model does contain useful information as shown in Figure
\ref{fig::CorrelationMarkov}.  Of course, there are extra edges due to
the induced correlations.  Nevertheless, most of the essential
structure is apparent.

We also considered the behavior of the correlation graph for a few
other models.  Figure \ref{fig::methods1} shows the correlation graph
for a null model, a dense covariance matrix, a four-block model and a
partial Markov chain (10 edges). In each case, $n=100$ and $D=12$.
Figure \ref{fig::methods2} shows the same models but with $D=200$.
For these models the method does very well even with $D>n$.

\begin{figure}
\begin{center}
\includegraphics[scale=.5]{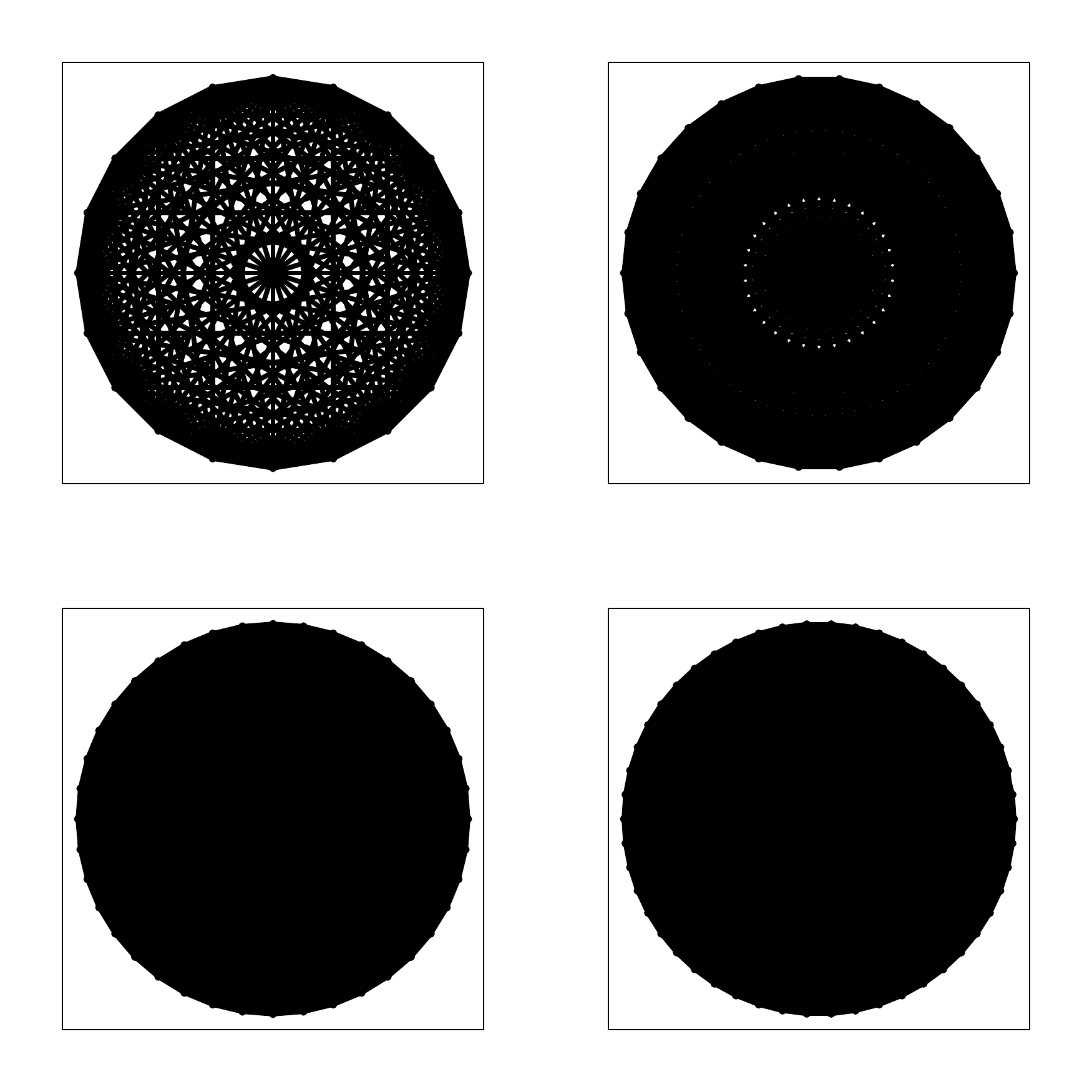}
\end{center}
\caption{Bootstrap based undirected graph for Dense model with
$\alpha =.9$, $a=.9$, $n=100$ and
dimensions 20,30,40,50.}
\label{fig::UndirectedDense}
\end{figure}

\begin{figure}
\begin{center}
\includegraphics[scale=.5]{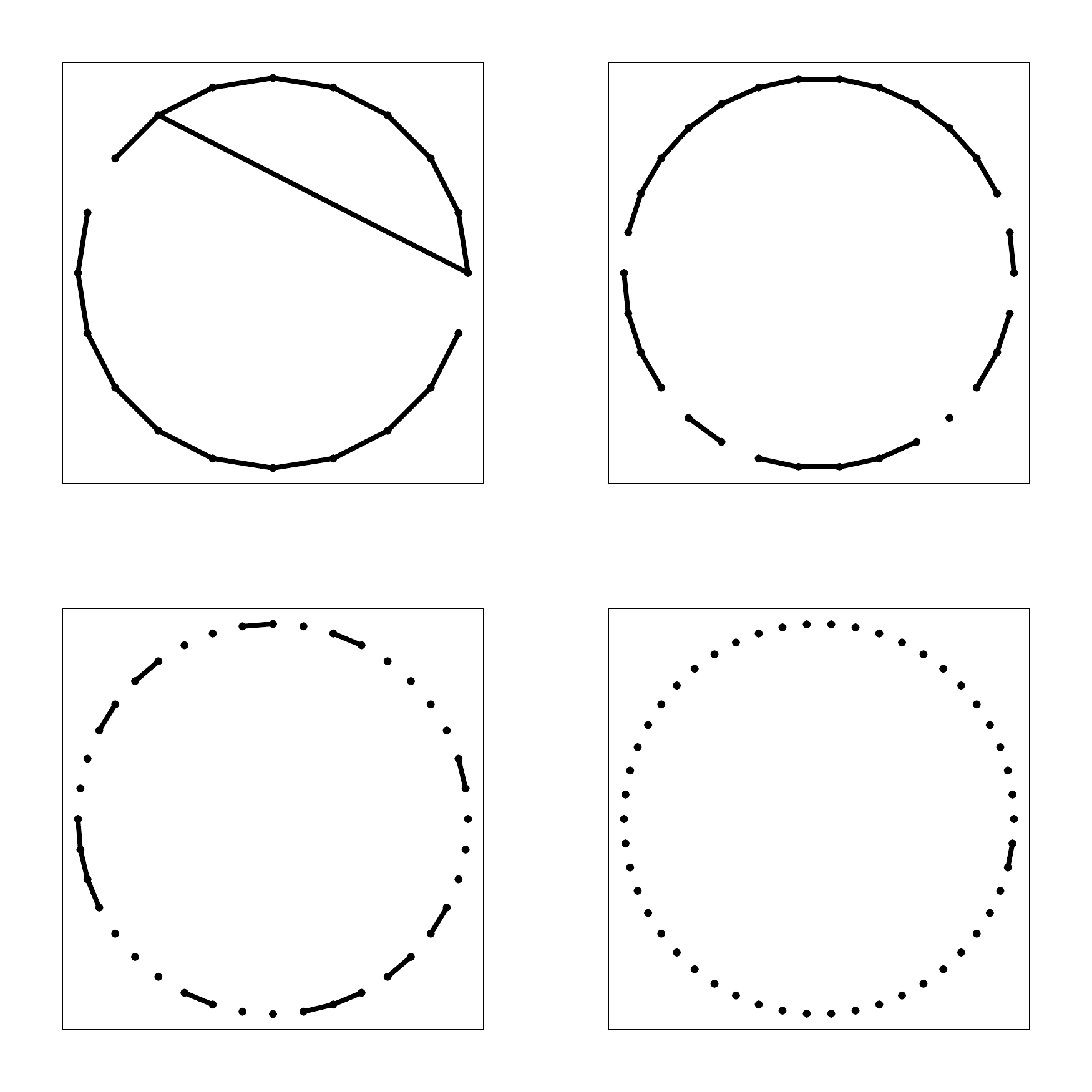}
\end{center}
\caption{Bootstrap based undirected graph for Markov model with
$\alpha =.9$, $a=.9$, $n=100$ and
dimensions 20,30,40,50.}
\label{fig::UndirectedMarkov}
\end{figure}

\begin{figure}
\begin{center}
\includegraphics[scale=.5]{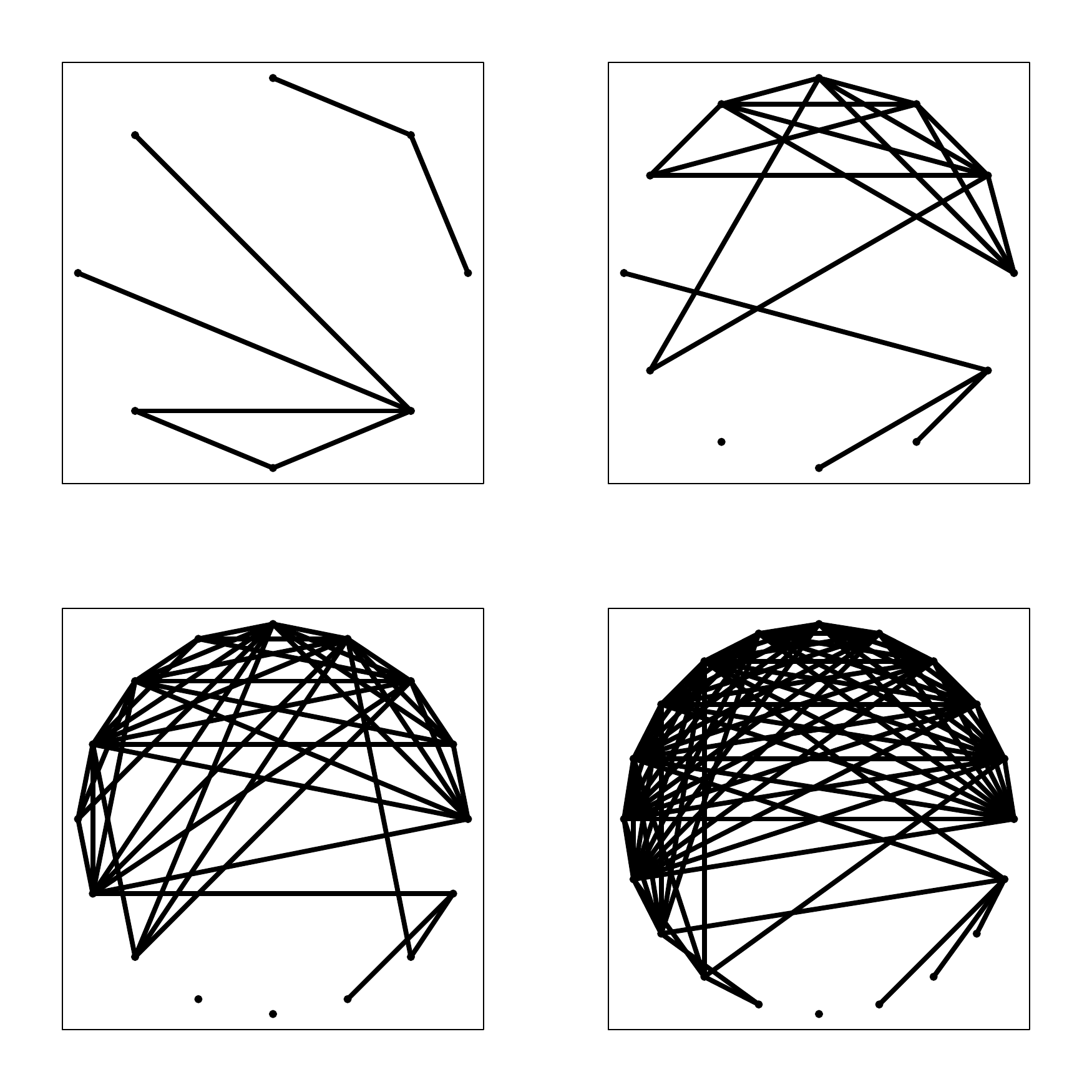}
\end{center}
\caption{Bootstrap based undirected graph for SEM model with
$\alpha =.9$, $a=.5$, $n=100$ and
dimensions 8,12,16,20.}
\label{fig::UndirectedSEM}
\end{figure}

\begin{figure}
\begin{center}
\includegraphics[scale=.5]{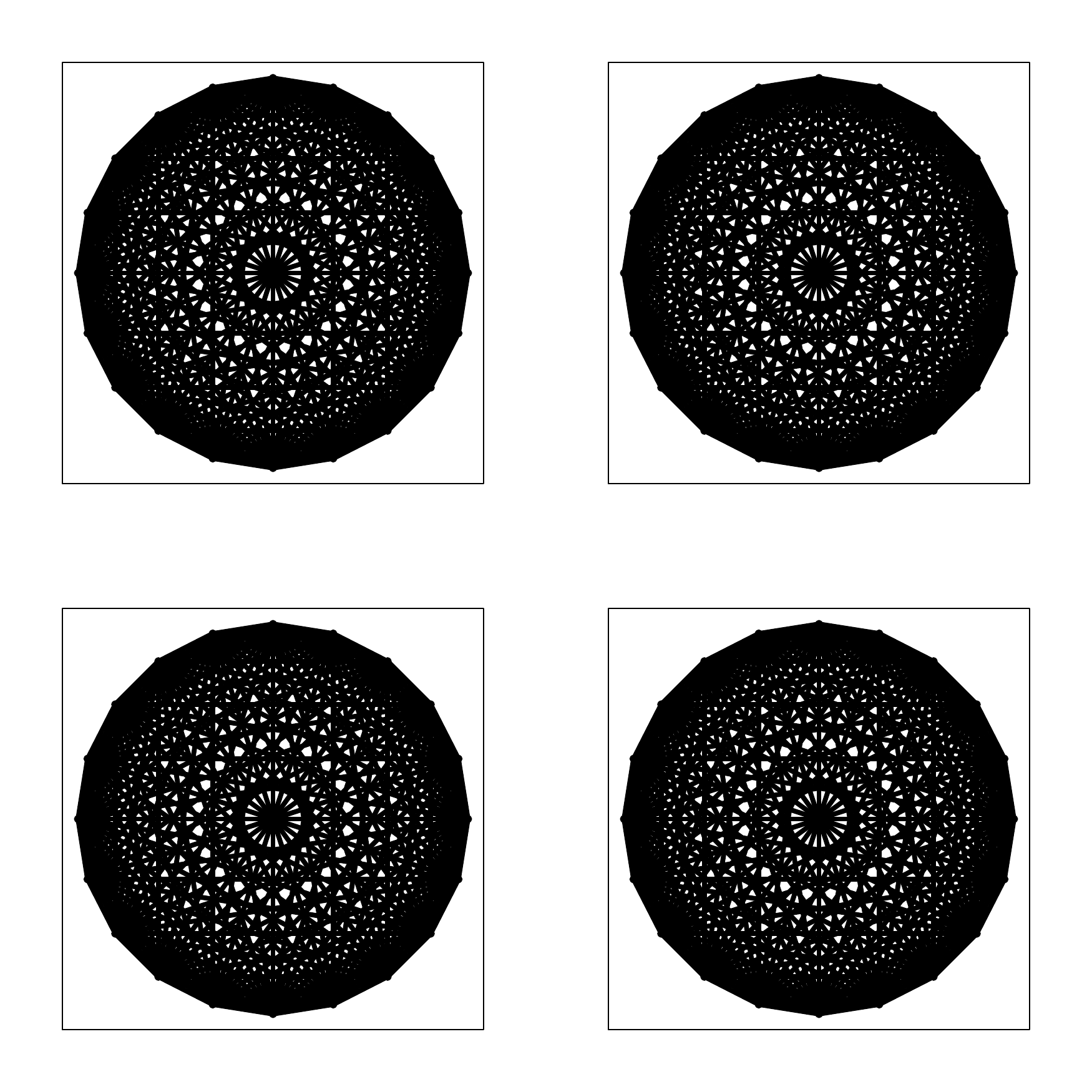}
\end{center}
\caption{Cluster graph for Dense model with
$\alpha =.9$, $a=.9$, $n=100$ and
dimensions 70, 80, 90, 100 and $L=20$.}
\label{fig::UndirectedDenseProto}
\end{figure}

\begin{figure}
\begin{center}
\includegraphics[scale=.5]{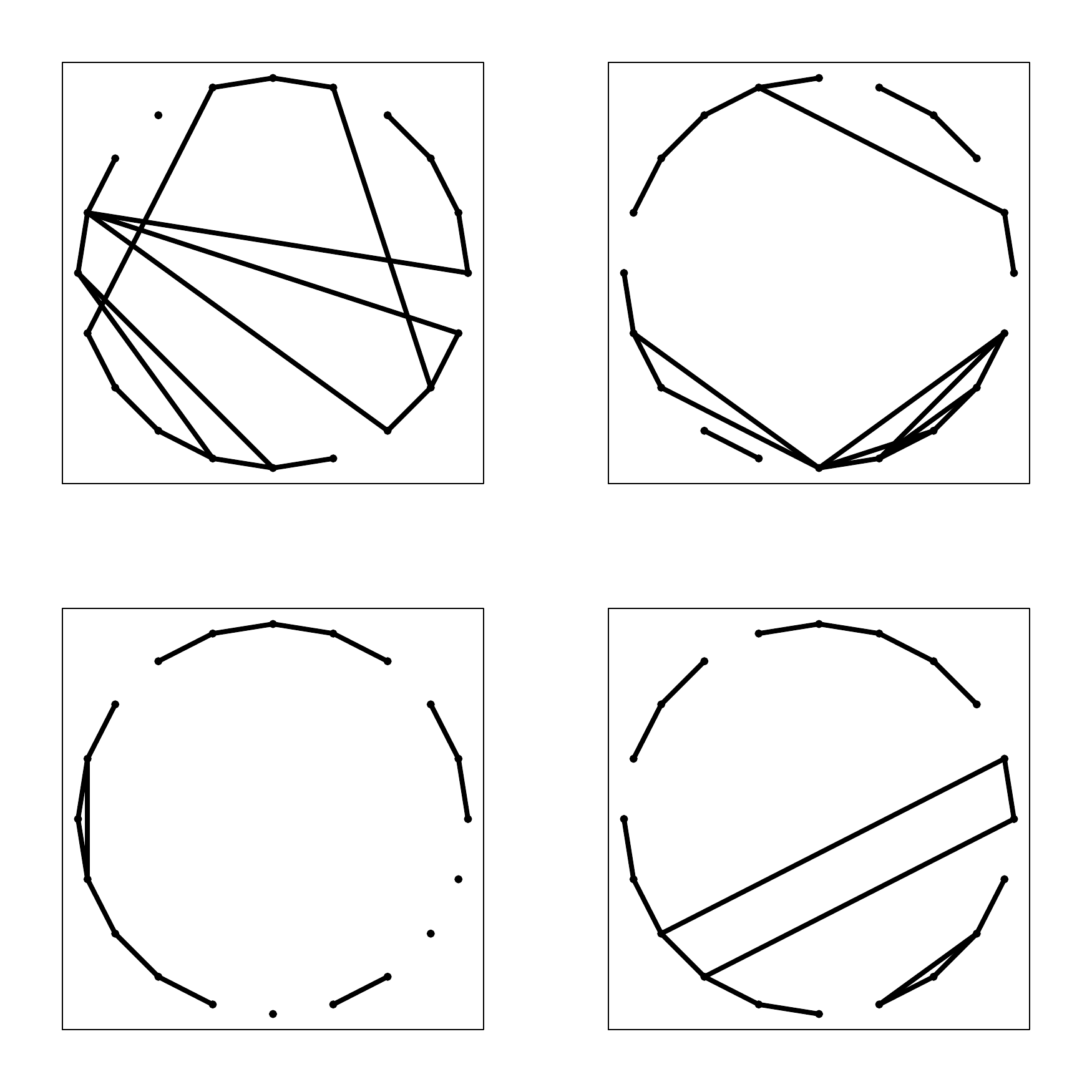}
\end{center}
\caption{Cluster graph for Markov model with
$\alpha =.9$, $a=.9$, $n=100$ and
dimensions 70, 80, 90, 100 and $L=20$.}
\label{fig::UndirectedMarkovProto}
\end{figure}

\begin{figure}
\begin{center}
\includegraphics[scale=.5]{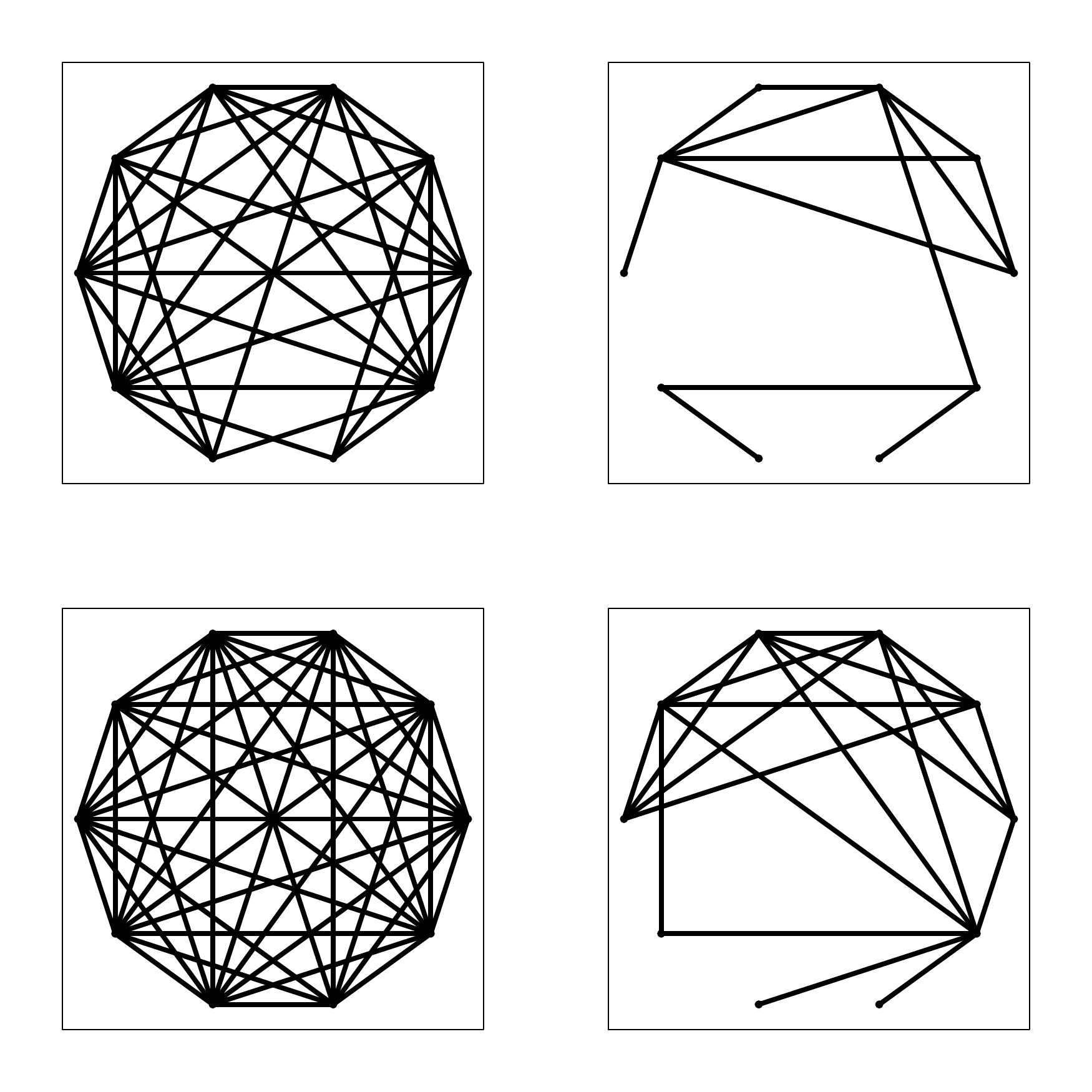}
\end{center}
\caption{Cluster graph for SEM model with
$\alpha =.9$, $a=.5$, $n=100$ and
dimensions 28, 32, 36, 40 and $L=10$.}
\label{fig::UndirectedSEMProto}
\end{figure}

\begin{figure}
\begin{center}
\includegraphics[scale=.5]{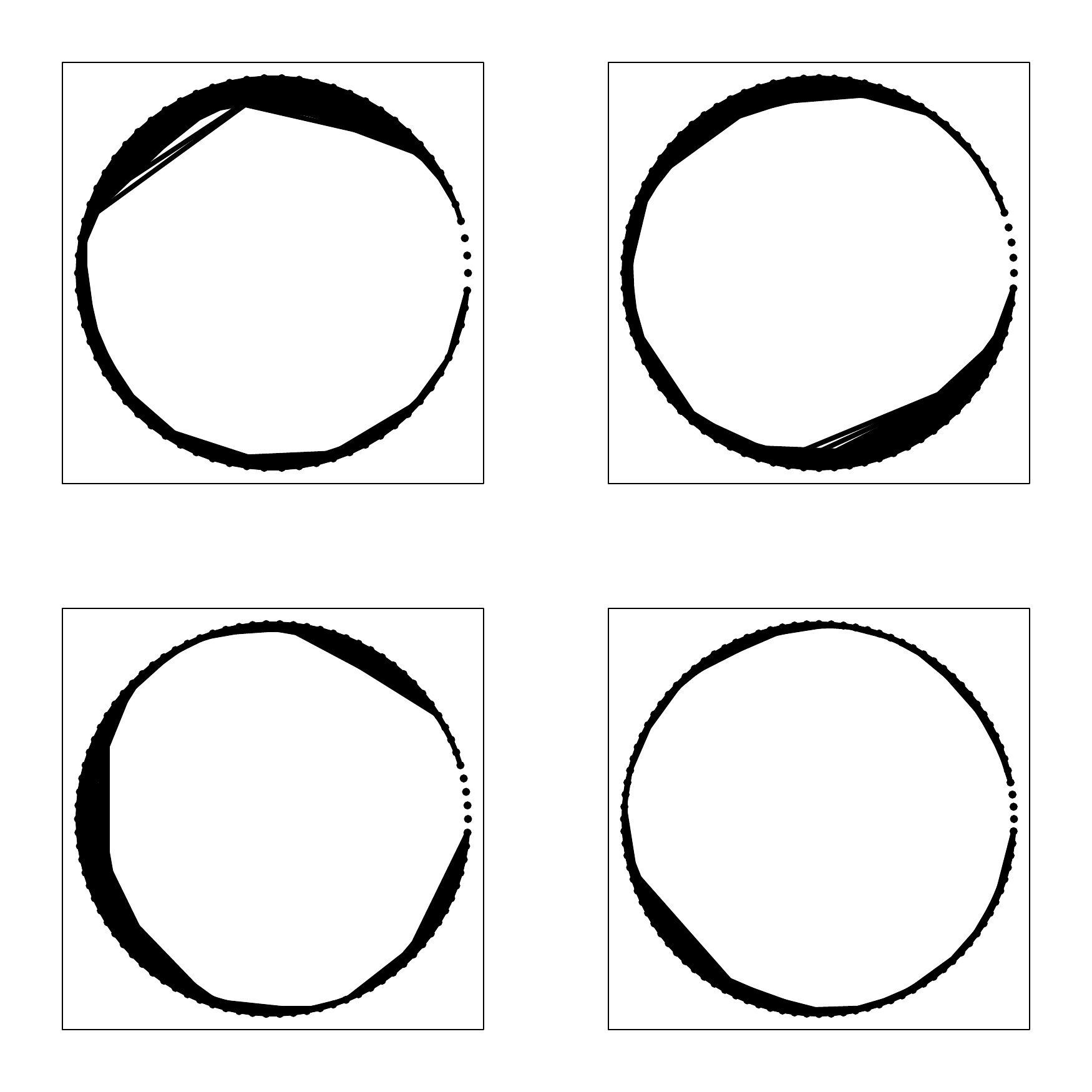}
\end{center}
\caption{Correlation graph for Markov model with
$\alpha =.9$, $a=.9$, $n=100$ and
dimensions 70, 80, 90, 100 and $L=20$.}
\label{fig::CorrelationMarkov}
\end{figure}

\begin{figure}
\begin{center}
\includegraphics[scale=.5]{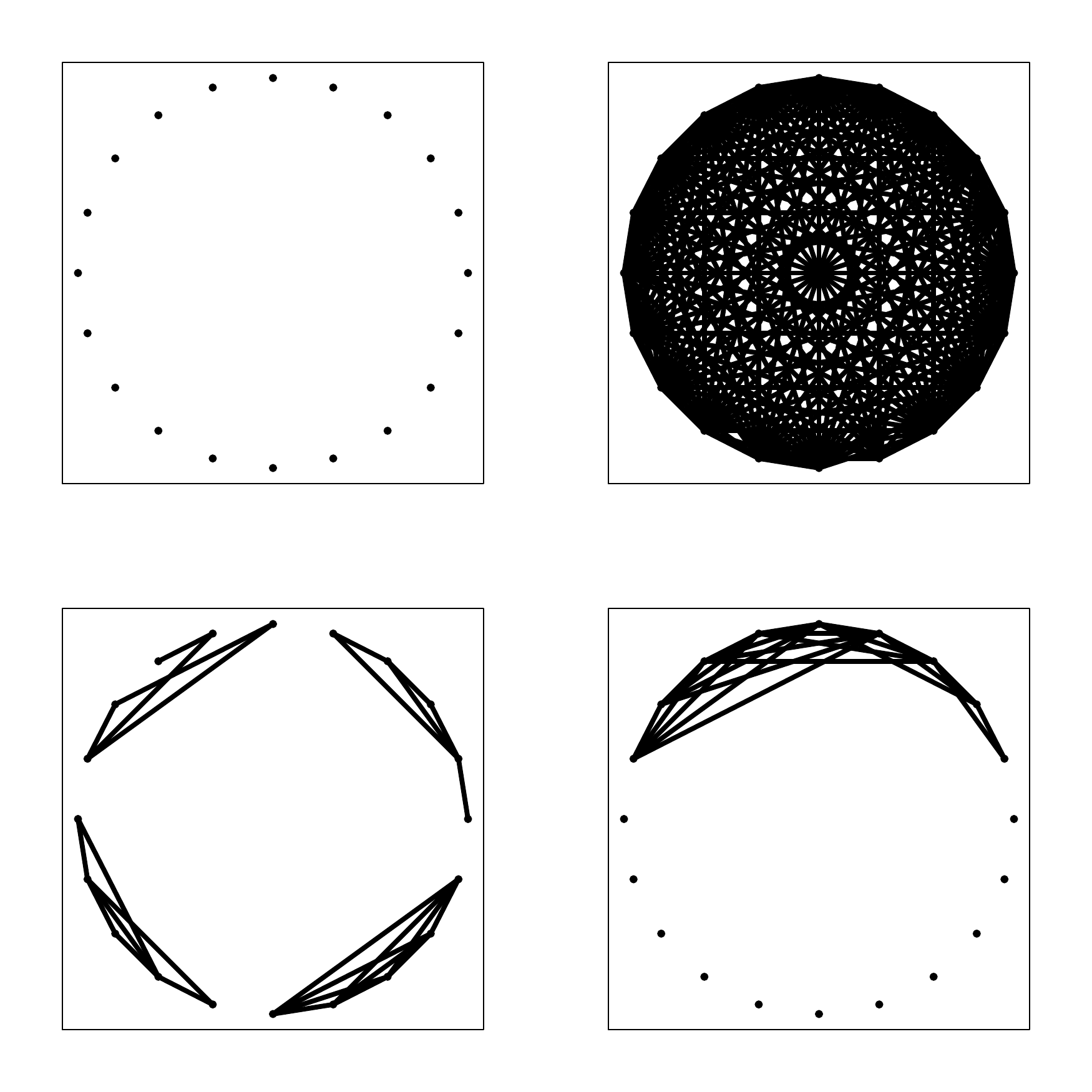}
\end{center}
\caption{Correlation Graphs, n=100, D=12.}
\label{fig::methods1}
\end{figure}

\begin{figure}
\begin{center}
\includegraphics[scale=.5]{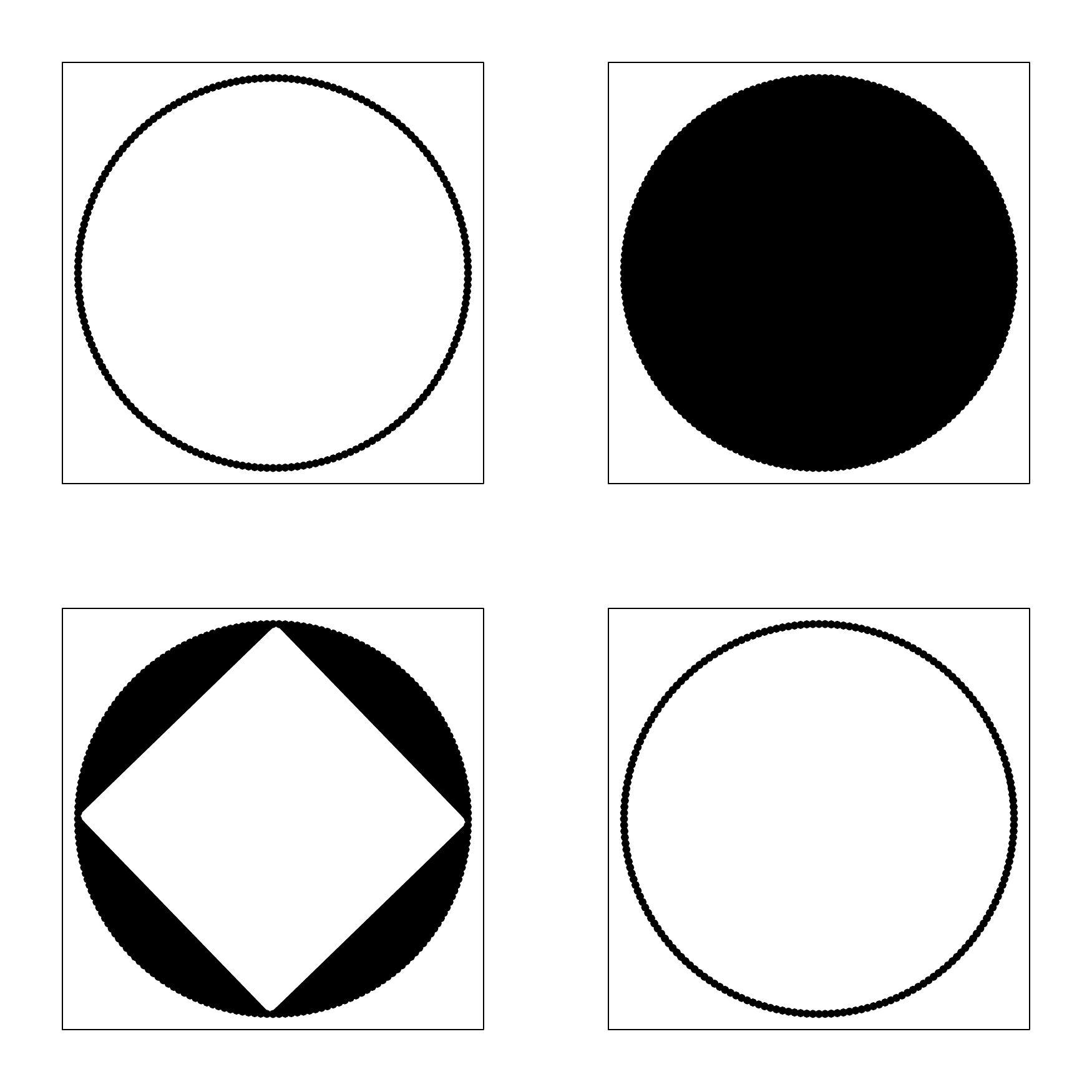}
\end{center}
\caption{Correlation Graphs, n=100, D=200.}
\label{fig::methods2}
\end{figure}

As mentioned earlier, the restricted partial correlation graph is so
computationally intensive that it is not yet practical.  We believe
the method is promising which is why we have included it in the paper
but at this point we do not have numerical experiments.

Finally, as a sanity check, we checked the coverage of the bootstrap
for two models: the null model (no edges) and the Markov model.  We
declare an error if there is even a single wrong edge.  Using $\alpha
= .10$ and $n=100$ we have the following error rates:

\begin{center}
\begin{tabular}{l|ll}
Model/Dimension   & D = 20 & D = 50 \\ \hline
Null              & .01    & .01\\
Markov            & .00    & .01\\
\end{tabular}
\end{center}

The error rates is well under $\alpha$.
Indeed, we see that the coverage is conservative
as we would expect.

\section{Conclusion}
\label{section::conclusion}

We have described methods for inferring graphs that use weak
assumptions and that have confidence guarantees.  Our methods are
atavistic: we use very traditional ideas that have been swept aside in
light of the newer sparsity-based approaches.  We do not mean in any
way to criticize sparsity-based methods which we find fascinating.
But our main message is that the older methods still have a role to
play especially if we want methods that use weaker assumptions.

There are several open problems that we will address in the future.
We briefly describe a few here.  First, we do not have any theory to
characterize how the original graph relates to the graph of the
dimension reduced problem.  It would be useful to have some general
theory which shows which features of the original graph are preserved.

Perhaps the most important extension is to go beyond linear measures
of dependence.  Following \cite{bergsma2011}, write
$$
Y = g(X) + \epsilon_Y\ \ \ {\rm and}\ \ \ 
Z = h(X) + \epsilon_Z
$$
and define the nonparametric partial correlation
$$
\theta_{YZ.X} =\frac{ \mathbb{E}(\epsilon_Y \epsilon_Z)} 
{\sqrt{\mathbb{E}(\epsilon_Y^2)\mathbb{E}(\epsilon_Z^2)}}.
$$
Let
$$
\hat\epsilon_{Y_i} = Y_i - \hat g(X_i)\ \ \ {\rm and}\ \ \ 
\hat\epsilon_{Z_i} = Y_i - \hat h(X_i).
$$
Let
$$
\hat\theta_{YZ.X} =
\frac{\sum_i \hat\epsilon_{Y_i} \hat\epsilon_{Z_i}}
{\sqrt{\sum_i \hat\epsilon_{Y_i}^2 \sum_i \hat\epsilon_{Z_i}^2}}.
$$
Bergsma shows that, for some $q_1,q_2>0$,
$$
\sqrt{n}(\hat\theta_{YZ.X}- \theta_{YZ.X}) =
\sqrt{n}(r_{YZ.X}- \theta_{YZ.X}) + O_P\left( n^{-{\rm min}(q_1,q_2)}\right)
$$
where
$$
r_{YZ.X} =
\frac{\sum_i \epsilon_{Y_i} \epsilon_{Z_i}}
{\sqrt{\sum_i \epsilon_{Y_i}^2 \sum_i \epsilon_{Z_i}^2}}
$$
and
$$
n^{q_1}(\hat g(x) - g(x)) = O_P(1),\ \ 
n^{q_2}(\hat h(x) - h(x)) = O_P(1).
$$
One can then extend the techniques in this paper
to get confidence measures.

Other problems for future development are: the development of
computationally efficient methods for computing the restricted partial
correlation graph and the extension of our theory to shrinkage graphs.

\section{Appendix: Alternative Delta Method}

If one is only interested in a single partial correlation,
then one can use
use a Taylor series together with the Berry-Esseen theorem.
We provide this analysis here.
At the end, we can turn this into 
a joint confidence set for all partial correlations using
the union bound but this leads to a larger error than our earlier analysis.
So the main interest of this section is single partial correlations.

Let us write $\theta_{jk} =
g_{jk}(\sigma)$ where $g_{jk}:\mathbb{R}^{D\times D}\to [-1,1]$.  Let
$\ell_{jk}$ and $H_{jk}$ denote the gradient and Hessian of $g_{jk}$.
Both $\ell_{jk}$ and $H_{jk}$ are bounded continuous functions as long
as $\Sigma$ is invertible.
The linearization of $\theta_{jk}$ is
\begin{equation}
\sqrt{n}(\hat\theta_{jk} - \theta_{jk}) =
\delta^T \ell_{jk} + \frac{R_{jk}}{\sqrt{n}}
\end{equation}
where $\ell_{jk} \equiv \ell_{jk}(\sigma)$ and the remainder term
$R_{jk}$ is
\begin{equation}
R_{jk} = \frac{1}{2} \delta^T H_{jk}(\tilde\sigma) \delta
\end{equation}
for some $\tilde\sigma$ between $\sigma$ and $s$.
We compute $\ell_{jk}$
and $H_{jk}$ 
explicitly in Section \ref{section::error}.

Let
$$
s_{jk}^2 = U(\sigma),\ \ \ \hat{s}_{jk}^2 = U(s)
$$
where
\begin{equation}
U_{jk}(\sigma) = \ell_{jk}(\sigma)^T T(\sigma) \ell_{jk}(\sigma).
\end{equation}
The asymptotic variance of
the linearized partial correlation $\delta^T \ell_{jk}$ is
$s_{jk}^2$ and its estimate is
$\hat{s}_{jk}^2$.

Define $B = \Bigl\{ a:\ ||a-\sigma|| \leq C\sqrt{D^2 \log
  n/n}\Bigr\}$.  It follows from Lemma \ref{lemma::ballbound} that,
for large enough $C$, $s\in B$ except on a set of probability at most
$1/n$.  Let
\begin{align*}
\xi_n    &= \sup_{a\in B}\max_{jk}||\ell_{jk}(a)||_1\\
\gamma_n &= \sup_{a\in B}\max_{jk} \sqrt{\frac{|\!|\!| H_{jk}(a) |\!|\!|}{s_{jk}(a)}}\\
\rho_n   &= \sup_{a\in B}\max_{jk} \frac{||Q'_{jk}(a)||_1}{s_{jk}}.
\end{align*}
Note that these constants are also functions of $D$.

We begin by approximating the distribution of a single partial
correlation.  Let
$$
T_{jk} = \frac{\sqrt{n}(\hat\theta_{jk}-\theta_{jk})}{s_{jk}}.
$$
We start by assuming that $s_{jk}^2 = \ell_{jk}(\sigma)^T T(\sigma) \ell_{jk}(\sigma)$
is known.

\begin{lemma}
We have
$$
\max_{j,k}\sup_z | P( T_{jk} \leq z) - \Phi(z)| \preceq \frac{1}{\sqrt{n}} +
 \frac{2\gamma_n}{\sqrt{n}} \log (n D^2).
$$
\end{lemma}

\begin{proof}
We have
$$
T_{jk} = \frac{U}{s_{jk}} + \frac{R_{jk}}{s_{jk}\sqrt{n}}
$$
where $U = \sqrt{n} a^T (s-\sigma) = n^{-1}\sum_i V_i$ where $V_i =
{\rm vec}(Y_i Y_i^T) - \sigma$ and $a = \ell_{jk}$.  By Lemma
\ref{lemma::simple}, for every $\epsilon>0$,
$$
\sup_z | P( T_{jk} \leq z) - \Phi(z)| \leq
\sup_z \left|P\left(\frac{U}{s_{jk}} \leq z\right) - \Phi(z) \right| +
\epsilon + P\left(\left|\frac{R_{jk}}{s_{jk}\sqrt{n}}\right| > \epsilon\right).
$$
Note that ${\rm Var}(V_i) = s_{jk}^2$ and
$$
\mathbb{E}|V_i|^3 \leq C \sum_i | a_j|^3.
$$
Let $Z\sim N(0,1)$.
By the Berry-Esseen theorem,
$$
\sup_t \Biggl| \mathbb{P}\left( \frac{U_n}{s_{jk}} \leq t\right) - P(Z \leq t) \Biggr| \preceq
\frac{ \sum_j |a_j|^3}{\sqrt{n} (a^T T a)^{3/2}} \leq
\frac{ \sum_j |a_j|^3}{\sqrt{n} c_0^{3/2} ||a||^3} \leq \frac{1}{\sqrt{n}}
$$
since
$||a||_3 \leq ||a||_2$ and
$\frac{\sum_j |a_j|^3}{||a||^3}  = \frac{||a||_3^3}{||a||_2^3}$.
Now
\begin{align*}
\left|\frac{R_{jk}}{s_{jk}\sqrt{n}}\right| =
\frac{1}{2} \frac{\delta^T H_{jk}\delta}{s_{jk}\sqrt{n}}  \leq
\frac{ \gamma_n ||\delta||_{\rm max}^2}{\sqrt{n}}.
\end{align*}
From Lemma \ref{lemma::simple},
\begin{align*}
P\left(\left|\frac{R_{jk}}{s_{jk}\sqrt{n}}\right| > \epsilon\right) & \leq
P\left(\frac{\gamma_n ||\delta||_{\rm max}^2}{\sqrt{n}} > \epsilon\right)=
P( ||s-\sigma||_\infty > \frac{\sqrt{\epsilon}}{n^{1/4} \sqrt{\gamma}})\\
& \leq
D^2 e^{-n \epsilon/(\gamma\sqrt{n})}.
\end{align*}
Let $\epsilon = \frac{\gamma}{\sqrt{n}} \log (n D^2).$ Then $D^2 e^{-n
  \epsilon/(\gamma\sqrt{n})} \leq \epsilon.$ The result follows.
\end{proof}

Now let
$$
Z_{jk} = \frac{\sqrt{n}(\hat\theta_{jk}-\theta_{jk})}{\hat s_{jk}}
$$
where
$\hat s_{jk}^2 = \ell_{jk}(s)^T T(s) \ell_{jk}(s)$.

\begin{theorem}
$$
\max_{j,k}\sup_z\left|
P\left( \frac{\sqrt{n}(\hat\theta_{jk}-\theta_{jk})}{\hat s_{jk}} \leq z \right) - \Phi(z)\right| \preceq
\sqrt{\frac{\rho_n}{n}}\log(n D^2) + \frac{\gamma_n}{\sqrt{n}}\log(n D^2).
$$
\end{theorem}

\begin{proof}
Let $E=\{ s_{jk}/\hat{s}_{jk} > 1+\epsilon\}$ and $F =\{ T_{jk} > u/\epsilon\}$ where 
$\epsilon = \sqrt{\rho_n/n}\log(n D^2)$ and $u = \epsilon \log(n)$.  
Note that $s_{jk} - \hat{s}_{jk} = U(\sigma) -U(s) = (\sigma-s)^T Q'$ where $Q'$ is the gradient of $Q$ evaluated
at some point between $s$ and $\sigma$.  Then, for $0 < \epsilon \leq 1$,
\begin{align*}
P(E^c) &= P\left( \frac{s_{jk}-\hat{s}_{jk}}{s_{jk}} > \frac{\epsilon}{1+\epsilon}\right)=
P\left( \frac{U(\sigma)-U(s)}{s_{jk}} > \frac{\epsilon}{1+\epsilon}\right)\\
&= P\left( \frac{ (\sigma-s)^T Q' }{s_{jk}} > \frac{\epsilon}{1+\epsilon}\right)\leq
P\left( \frac{||s-\sigma||_\infty ||Q'||_1}{s_{jk}}> \frac{\epsilon}{1+\epsilon}\right)\\
& \leq P\left( ||s-\sigma||_\infty \rho_n > \frac{\epsilon}{1+\epsilon}\right)=
P\left( ||s-\sigma||_\infty > \frac{\epsilon}{2\rho_n}\right)\\
& \leq D^2 e^{-n\epsilon^2/(4\rho_n^2)} \leq \epsilon.
\end{align*}
Now,
\begin{align*}
P\left( \frac{\sqrt{n}(\hat\theta_{jk}-\theta_{jk})}{\hat s_{jk}} \leq z \right) &- \Phi(z)  =
P\left(T_{jk}\left(\frac{s_{jk}}{\hat s_{jk}}\right) \leq z \right) - \Phi(z)\\
& \leq P\left(T_{jk} (1-\epsilon) \leq z \right) +P(E^c) - \Phi(z)\\
&= P\left(T_{jk} -T_{jk}\epsilon) \leq z \right) +P(E^c) - \Phi(z)\\
&\leq P\left(T_{jk}  \leq z+u \right) + P(F^c) + P(E^c) - \Phi(z)\\
&\leq P\left(T_{jk}  \leq z+u \right) -\Phi(z+u) + P(F^c) + P(E^c) +u\\
& \leq P\left(T_{jk}  \leq z+u \right) -\Phi(z+u) + P(F^c) + \epsilon +u.
\end{align*}
Now
\begin{align*}
P(F^c) &=P( T_{jk}> u/\epsilon) \leq
P(Z > u/\epsilon) + \frac{\gamma_n}{\sqrt{n}}\log(n D^2)\\
&=
P(Z > \log n) + 
\frac{\gamma_n}{\sqrt{n}}\log(n D^2)\\
& \preceq \frac{\gamma_n}{\sqrt{n}}\log(n D^2).
\end{align*}
So,
\begin{align*}
P\left( \frac{\sqrt{n}(\hat\theta_{jk}-\theta_{jk})}{\hat s_{jk}} \leq z \right) - \Phi(z)  &\leq
P\left(T_{jk}  \leq z+u \right) -\Phi(z+u)  + \epsilon +u+ \frac{1}{n} + \frac{\gamma_n}{\sqrt{n}}\log(n D^2)\\
& \preceq
\sqrt{\frac{\rho_n}{n}}\log(n D^2) + \frac{\gamma_n}{\sqrt{n}}\log(n D^2).
\end{align*}
Taking the supremum over $z$ gives an upper.  A similar lower bound
completes the proof.
\end{proof}

Now we turn to bounding
$P(\max_{jk}|Z_{jk}| > z)$.
We use the union bound.
So,
\begin{align*}
P(\max_{jk}|Z_{jk}| > z) &\leq
\sum_{jk} P(|Z_{jk}| > z) \\
&=
D^2 \Phi(z) + \sum_{jk} [P(|Z_{jk}| > z) -\Phi(z)]\\
& \leq
D^2 \Phi(z) +
D^2\left[\sqrt{\frac{\rho_n}{n}}\log(n D^2) + \frac{\gamma_n}{\sqrt{n}}\log(n D^2)\right].
\end{align*}
Setting $z= -\Phi(\alpha/D^2)$ we have that
\begin{align*}
P(\max_{jk}|Z_{jk}| > z) & \leq 
\alpha+
D^2\left[\sqrt{\frac{\rho_n}{n}}\log(n D^2) + \frac{\gamma_n}{\sqrt{n}}\log(n D^2)\right].
\end{align*}

\begin{corollary}
Let $z\equiv z_{\alpha/D^2}$ and let
$$
R =\bigotimes_{j,k} \Bigl[\hat\theta_{jk} - \frac{z \hat s_{jk}}{\sqrt{n}},\ \hat\theta_{jk} + \frac{z \hat s_{jk}}{\sqrt{n}}\Bigr].
$$
Then
$$
P(\theta\in R) = 1-\alpha + 
D^2\left[\sqrt{\frac{\rho_n}{n}}\log(n D^2) + \frac{\gamma_n}{\sqrt{n}}\log(n D^2)\right].
$$
\end{corollary}

Note the presence of the $D^2$ term.
This term is avoided in the analysis in Section 
\ref{section::moderate}.

\bibliography{paper}

\begin{thebibliography}{21}

\bibitem[\protect\citeauthoryear{Bergsma}{2011}]{bergsma2011}
\begin{barticle}[author]
\bauthor{\bsnm{Bergsma},~\bfnm{Wicher}\binits{W.}}
(\byear{2011}).
\btitle{A note on the distribution of the partial correlation coefficient with
  nonparametrically estimated marginal regressions}.
\bjournal{arXiv:1101.4616}.
\end{barticle}
\endbibitem

\bibitem[\protect\citeauthoryear{Boik and Haaland}{2006}]{boik2006second}
\begin{barticle}[author]
\bauthor{\bsnm{Boik},~\bfnm{RJ}\binits{R.}} \AND
  \bauthor{\bsnm{Haaland},~\bfnm{B.}\binits{B.}}
(\byear{2006}).
\btitle{Second-order accurate inference on simple, partial, and multiple
  correlations}.
\bjournal{Journal of Modern Applied Statistical Methods}
\bvolume{5}
\bpages{283--308}.
\end{barticle}
\endbibitem

\bibitem[\protect\citeauthoryear{Castelo and
  Roverato}{2006}]{castelo2006robust}
\begin{barticle}[author]
\bauthor{\bsnm{Castelo},~\bfnm{Robert}\binits{R.}} \AND
  \bauthor{\bsnm{Roverato},~\bfnm{Alberto}\binits{A.}}
(\byear{2006}).
\btitle{A robust procedure for Gaussian graphical model search from microarray
  data with p larger than n}.
\bjournal{The Journal of Machine Learning Research}
\bvolume{7}
\bpages{2621--2650}.
\end{barticle}
\endbibitem

\bibitem[\protect\citeauthoryear{Chen and Shao}{2007}]{chen2007normal}
\begin{barticle}[author]
\bauthor{\bsnm{Chen},~\bfnm{Louis~HY}\binits{L.~H.}} \AND
  \bauthor{\bsnm{Shao},~\bfnm{Qi-Man}\binits{Q.-M.}}
(\byear{2007}).
\btitle{Normal approximation for nonlinear statistics using a concentration
  inequality approach}.
\bjournal{Bernoulli}
\bpages{581--599}.
\end{barticle}
\endbibitem

\bibitem[\protect\citeauthoryear{Chernozhukov, Chetverikov and
  Kato}{2012}]{Cherno}
\begin{barticle}[author]
\bauthor{\bsnm{Chernozhukov},~\bfnm{Victor}\binits{V.}},
  \bauthor{\bsnm{Chetverikov},~\bfnm{Denis}\binits{D.}} \AND
  \bauthor{\bsnm{Kato},~\bfnm{Kengo}\binits{K.}}
(\byear{2012}).
\btitle{Central Limit Theorem and Multiplier Boostrap When $p$ is Much Larger
  Than $n$}.
\bjournal{arXiv:1212.6906}.
\end{barticle}
\endbibitem

\bibitem[\protect\citeauthoryear{Chernozhukov, Chetverikov and
  Kato}{2013}]{Cherno2}
\begin{barticle}[author]
\bauthor{\bsnm{Chernozhukov},~\bfnm{Victor}\binits{V.}},
  \bauthor{\bsnm{Chetverikov},~\bfnm{Denis}\binits{D.}} \AND
  \bauthor{\bsnm{Kato},~\bfnm{Kengo}\binits{K.}}
(\byear{2013}).
\btitle{Comparison and anti-concentration bounds for maxima of Gaussian random
  vectors}.
\bjournal{arXiv:1301.4807}.
\end{barticle}
\endbibitem

\bibitem[\protect\citeauthoryear{Drton and Perlman}{2004}]{drton2004model}
\begin{barticle}[author]
\bauthor{\bsnm{Drton},~\bfnm{Mathias}\binits{M.}} \AND
  \bauthor{\bsnm{Perlman},~\bfnm{Michael~D}\binits{M.~D.}}
(\byear{2004}).
\btitle{Model selection for Gaussian concentration graphs}.
\bjournal{Biometrika}
\bvolume{91}
\bpages{591--602}.
\end{barticle}
\endbibitem

\bibitem[\protect\citeauthoryear{Friedman and
  Tibshirani}{2007}]{friedman2007graphical}
\begin{barticle}[author]
\bauthor{\bsnm{Friedman},~\bfnm{J.}\binits{J.}} \AND
  \bauthor{\bsnm{Tibshirani},~\bfnm{R.}\binits{R.}}
(\byear{2007}).
\btitle{Graphical lasso}.
\end{barticle}
\endbibitem

\bibitem[\protect\citeauthoryear{Harris and Drton}{2012}]{harris2012pc}
\begin{barticle}[author]
\bauthor{\bsnm{Harris},~\bfnm{N.}\binits{N.}} \AND
  \bauthor{\bsnm{Drton},~\bfnm{M.}\binits{M.}}
(\byear{2012}).
\btitle{PC algorithm for Gaussian copula graphical models}.
\bjournal{arXiv preprint arXiv:1207.0242}.
\end{barticle}
\endbibitem

\bibitem[\protect\citeauthoryear{Horn and Johnson}{1990}]{horn1990matrix}
\begin{bbook}[author]
\bauthor{\bsnm{Horn},~\bfnm{R.~A.}\binits{R.~A.}} \AND
  \bauthor{\bsnm{Johnson},~\bfnm{C.~R.}\binits{C.~R.}}
(\byear{1990}).
\btitle{Matrix analysis}.
\bpublisher{Cambridge university press}.
\end{bbook}
\endbibitem

\bibitem[\protect\citeauthoryear{Ledoit and Wolf}{2004}]{ledoit2004well}
\begin{barticle}[author]
\bauthor{\bsnm{Ledoit},~\bfnm{Olivier}\binits{O.}} \AND
  \bauthor{\bsnm{Wolf},~\bfnm{Michael}\binits{M.}}
(\byear{2004}).
\btitle{A well-conditioned estimator for large-dimensional covariance
  matrices}.
\bjournal{Journal of multivariate analysis}
\bvolume{88}
\bpages{365--411}.
\end{barticle}
\endbibitem

\bibitem[\protect\citeauthoryear{Liu}{2013}]{WLiu}
\begin{barticle}[author]
\bauthor{\bsnm{Liu},~\bfnm{Weidong}\binits{W.}}
(\byear{2013}).
\btitle{Gaussian Graphical Model Estimation With False Discovery Rate Control}.
\bjournal{arXiv preprint arXiv:1306.0976}.
\end{barticle}
\endbibitem

\bibitem[\protect\citeauthoryear{Magnus and Neudecker}{1988}]{magnus1988matrix}
\begin{barticle}[author]
\bauthor{\bsnm{Magnus},~\bfnm{X}\binits{X.}} \AND
  \bauthor{\bsnm{Neudecker},~\bfnm{Heinz}\binits{H.}}
(\byear{1988}).
\btitle{Matrix differential calculus}.
\bjournal{New York}.
\end{barticle}
\endbibitem

\bibitem[\protect\citeauthoryear{Mammen}{1993}]{mammen1993bootstrap}
\begin{barticle}[author]
\bauthor{\bsnm{Mammen},~\bfnm{Enno}\binits{E.}}
(\byear{1993}).
\btitle{Bootstrap and wild bootstrap for high dimensional linear models}.
\bjournal{The Annals of Statistics}
\bpages{255--285}.
\end{barticle}
\endbibitem

\bibitem[\protect\citeauthoryear{Meinshausen and
  B{\"u}hlmann}{2006}]{meinshausen2006high}
\begin{barticle}[author]
\bauthor{\bsnm{Meinshausen},~\bfnm{N.}\binits{N.}} \AND
  \bauthor{\bsnm{B{\"u}hlmann},~\bfnm{P.}\binits{P.}}
(\byear{2006}).
\btitle{High-dimensional graphs and variable selection with the lasso}.
\bjournal{The Annals of Statistics}
\bvolume{34}
\bpages{1436--1462}.
\end{barticle}
\endbibitem

\bibitem[\protect\citeauthoryear{Pinelis and Molzon}{2013}]{Pinelis}
\begin{barticle}[author]
\bauthor{\bsnm{Pinelis},~\bfnm{Iosif}\binits{I.}} \AND
  \bauthor{\bsnm{Molzon},~\bfnm{Raymond}\binits{R.}}
(\byear{2013}).
\btitle{Berry-Esseen bounds for general nonlinear statistics, with applications
  to Pearson's and non-central Student's and Hotelling's}.
\bjournal{arXiv preprint arXiv:0906.0177}.
\end{barticle}
\endbibitem

\bibitem[\protect\citeauthoryear{Portnoy}{1988}]{portnoy1988asymptotic}
\begin{barticle}[author]
\bauthor{\bsnm{Portnoy},~\bfnm{Stephen}\binits{S.}}
(\byear{1988}).
\btitle{Asymptotic behavior of likelihood methods for exponential families when
  the number of parameters tends to infinity}.
\bjournal{The Annals of Statistics}
\bvolume{16}
\bpages{356--366}.
\end{barticle}
\endbibitem

\bibitem[\protect\citeauthoryear{Ren et~al.}{2013}]{Ren}
\begin{barticle}[author]
\bauthor{\bsnm{Ren},~\bfnm{Zhao}\binits{Z.}},
  \bauthor{\bsnm{Sun},~\bfnm{Tingni}\binits{T.}},
  \bauthor{\bsnm{Zhange},~\bfnm{Cun-Hui}\binits{C.-H.}} \AND
  \bauthor{\bsnm{Zhou},~\bfnm{Harrison}\binits{H.}}
(\byear{2013}).
\btitle{Asymptotic normality and optimalities in estimation of large Gaissian
  graphical models}.
\bjournal{manuscript}.
\end{barticle}
\endbibitem

\bibitem[\protect\citeauthoryear{Sch{\"a}fer
  et~al.}{2005}]{schafer2005shrinkage}
\begin{barticle}[author]
\bauthor{\bsnm{Sch{\"a}fer},~\bfnm{Juliane}\binits{J.}},
  \bauthor{\bsnm{Strimmer},~\bfnm{Korbinian}\binits{K.}} \betal{et~al.}
(\byear{2005}).
\btitle{A shrinkage approach to large-scale covariance matrix estimation and
  implications for functional genomics}.
\bjournal{Statistical applications in genetics and molecular biology}
\bvolume{4}
\bpages{32}.
\end{barticle}
\endbibitem

\bibitem[\protect\citeauthoryear{Vershynin}{2010}]{vershynin2010introduction}
\begin{barticle}[author]
\bauthor{\bsnm{Vershynin},~\bfnm{Roman}\binits{R.}}
(\byear{2010}).
\btitle{Introduction to the non-asymptotic analysis of random matrices}.
\bjournal{arXiv preprint arXiv:1011.3027}.
\end{barticle}
\endbibitem

\bibitem[\protect\citeauthoryear{Yuan and Lin}{2007}]{yuan2007model}
\begin{barticle}[author]
\bauthor{\bsnm{Yuan},~\bfnm{M.}\binits{M.}} \AND
  \bauthor{\bsnm{Lin},~\bfnm{Y.}\binits{Y.}}
(\byear{2007}).
\btitle{Model selection and estimation in the Gaussian graphical model}.
\bjournal{Biometrika}
\bvolume{94}
\bpages{19--35}.
\end{barticle}
\endbibitem

\end{thebibliography}

\end{document}